\documentclass[oneside,english]{amsart}
\usepackage[T1]{fontenc}
\usepackage[latin9]{inputenc}
\usepackage{babel}
\usepackage{amsthm}
\usepackage{amsbsy}
\usepackage{amssymb}
\usepackage{xargs}[2008/03/08]
\usepackage[unicode=true,pdfusetitle,
 bookmarks=true,bookmarksnumbered=false,bookmarksopen=false,
 breaklinks=false,pdfborder={0 0 1},backref=false,colorlinks=false]
 {hyperref}
\usepackage{breakurl}

\makeatletter
\numberwithin{equation}{section}
\numberwithin{figure}{section}
\theoremstyle{plain}
\newtheorem{thm}{\protect\theoremname}[section]
  \theoremstyle{definition}
  \newtheorem{defn}[thm]{\protect\definitionname}
  \theoremstyle{plain}
  \newtheorem{prop}[thm]{\protect\propositionname}
  \theoremstyle{remark}
  \newtheorem{rem}[thm]{\protect\remarkname}
  \theoremstyle{definition}
  \newtheorem{example}[thm]{\protect\examplename}
  \theoremstyle{plain}
  \newtheorem{lem}[thm]{\protect\lemmaname}
  \theoremstyle{plain}
  \newtheorem{cor}[thm]{\protect\corollaryname}
\newenvironment{lyxlist}[1]
{\begin{list}{}
{\settowidth{\labelwidth}{#1}
 \setlength{\leftmargin}{\labelwidth}
 \addtolength{\leftmargin}{\labelsep}
 }}
{\end{list}}

\usepackage{MnSymbol}
\usepackage{tikz}
\usetikzlibrary{positioning}
\usetikzlibrary{matrix,arrows}
\makeatother

  \providecommand{\corollaryname}{Corollary}
  \providecommand{\definitionname}{Definition}
  \providecommand{\examplename}{Example}
  \providecommand{\lemmaname}{Lemma}
  \providecommand{\propositionname}{Proposition}
  \providecommand{\remarkname}{Remark}
\providecommand{\theoremname}{Theorem}

\begin{document}

\title{TANNAKA DUALITY, COCLOSED CATEGORIES AND RECONSTRUCTION FOR NONARCHIMEDEAN
BIALGEBRAS}

\author{Anton Lyubinin}

\address{Department of Mathematics, School of Mathematical Sciences,}

\address{University of Science and Technology of China, Hefei, Anhui, PRC}

\email{anton@lyubinin.kiev.ua}
\begin{abstract}
The topic of this paper is a generalization of Tannaka duality to
coclosed categories. As an application we prove reconstruction theorems
for coalgebras (bialgebras, Hopf algebras) in categories of topological
vector spaces over a nonarchimedean field $K$. In particular, our
results imply reconstruction and recognition theorems for categories
of locally analytic representations of compact $p$-adic groups, which
was the major motivation for this work. Also, as an example, we discuss
a certain (trivial) extension of the geometric Satake correspondence.
\end{abstract}
\maketitle
\global\long\def\Imm#1{\mathrm{Im}\left(#1\right)}
\global\long\def\Id{\mathrm{Id}}
\global\long\def\Ker#1{\mathrm{Ker}\left(#1\right)}

\global\long\def\Bann{\mathrm{Ban}_{K}^{\infty}}
\global\long\def\Banc{\mathrm{Ban}_{K}^{\leq1}}
\global\long\def\Ban{\mathrm{Ban}_{K}}

\global\long\def\BAlgn{\mathrm{BAlg}_{K}^{\infty}}
\global\long\def\BAlgc{\mathrm{BAlg}_{K}^{\leq1}}
\global\long\def\BAlg{\mathrm{BAlg}_{K}}

\global\long\def\BCAlgn{\mathrm{BCoalg}_{K}^{\infty}}
\global\long\def\BCAlgc{\mathrm{BCoalg}_{K}^{\leq1}}
\global\long\def\BCAlg{\mathrm{BCoalg}_{K}}

\global\long\def\Bmod#1{\mathrm{BMod}_{#1}}
\global\long\def\Bmodc{\mathrm{Bmod}_{K}^{\leq1}}

\global\long\def\rctcomod#1{\mathrm{CTComod}_{#1}}
\global\long\def\lctcomod#1{\mathrm{_{#1}CTComod}}

\global\long\def\rlscomod#1{\mathrm{LSComod}_{#1}}
\global\long\def\llscomod#1{\mathrm{_{#1}LSComod}}

\global\long\def\rdlsmod#1{\mathrm{FSMod}_{#1}}
\global\long\def\ldlsmod#1{\mathrm{_{#1}FSMod}}

\global\long\def\rnfmod#1{\mathrm{NFMod}_{#1}}
\global\long\def\lnfmod#1{\mathrm{_{#1}NFMod}}

\global\long\def\rbcom#1{\mathrm{BComod}_{#1}}
\global\long\def\rbcomn#1{\mathrm{BComod}_{#1}^{\infty}}
\global\long\def\rbcomc#1{\mathrm{BComod}_{#1}^{\leq1}}

\global\long\def\lbcom#1{\!{}_{#1}\mathrm{BComod}}
\global\long\def\lbcomn#1{\!{}_{#1}\mathrm{BComod}^{\infty}}
\global\long\def\lbcomc#1{\!{}_{#1}\mathrm{BComod}^{\leq1}}

\global\long\def\rbmod#1{\mathrm{BMod}_{#1}}
\global\long\def\lbmod#1{\!_{#1}\mathrm{BMod}}

\global\long\def\uball#1{B_{#1}}

\global\long\def\sd#1{\left(#1\right)'_{b}}
\global\long\def\wdual#1{\left(#1\right)'_{s}}
\global\long\def\bd#1{#1'}
\global\long\def\bhd#1{#1^{\odot}}

\global\long\def\hits{\rightharpoonup}
\global\long\def\hit{\leftharpoonup}

\global\long\def\cten{\widehat{\otimes}}
\global\long\def\ten{\otimes}
\global\long\def\ep{\epsilon}
\global\long\def\emb{\hookrightarrow}
\global\long\def\dperp{\upmodels}
\global\long\def\from{\leftarrow}

\global\long\def\cm{\Delta}
\global\long\def\bten{\bar{\otimes}}
 \newcommandx\norm[1][usedefault, addprefix=\global, 1=\cdot]{\left\Vert #1\right\Vert }
\global\long\def\ccten#1{\underset{#1}{\widehat{\boxtimes}}}
\global\long\def\hd#1{#1^{\widehat{0}}}

\global\long\def\id#1{\mathrm{id}_{#1}}
\global\long\def\ann#1{\left(#1\right)^{\perp}}

\global\long\def\ilim{{\displaystyle \lim_{\longrightarrow}\,}}
\global\long\def\plim{{\displaystyle \lim_{\longleftarrow}\,}}

\newcommandx\compl[2][usedefault, addprefix=\global, 1=]{#2^{\wedge#1}}

\global\long\def\cal#1{\mathcal{#1}}
\global\long\def\mc#1{\mathcal{#1}}

\global\long\def\funct#1{\mathbb{#1}}
\global\long\def\rl#1{\underline{\mathrm{#1}}}

\global\long\def\Cten#1{\otimes_{\mc{#1}}}

\global\long\def\coend#1#2{{\rm coend}_{\cal{\mc{#1}}}\left(#2\right)}
\global\long\def\coendf#1#2{{\rm coend}_{#2}\left(#1\right)}
\newcommandx\cohom[3][usedefault, addprefix=\global, 1=]{{\rm cohom}_{\cal{#1}}\left(#2,#3\right)}
\newcommandx\cohomm[3][usedefault, addprefix=\global, 1=]{{\rm cohom}_{#1}\left(#2,#3\right)}
\global\long\def\coev#1{{\rm coev}_{#1}}
\global\long\def\blim{\mathfrak{B}-\lim}
\global\long\def\bcolim{\mathfrak{B}-\mathrm{colim}}

\global\long\def\nat#1#2#3{{\rm Nat}_{#1}\left(#2,#3\right)}

\global\long\def\comodc#1{{\rm Comod}_{\cal{#1}}}
\global\long\def\comod#1{{\rm Comod}_{#1}}

\section*{Introduction}

The first results on reconstruction seems to be the Pontryagin duality
theorems for abelian locally compact groups, followed by the T. Tannaka
reconstruction theorem \cite{Ta} and M. Krein recognition theorem
\cite{Kr} for general compact groups. In short, Tannaka-Krein duality
states that one can reconstruct the group from its category of finite-dimensional
representations. It was then extended to pro-algebraic groups and
Hopf algebras in \cite{D,DM,S}, where the concepts of rigid objects
and Tannakian categories were introduced. Around the same time reconstruction
theorems for infinite-dimensional representations (in vector space
over the field of complex numbers) of (locally) compact groups were
obtained in several papers, among which we would like to mention \cite{R1,R2},
followed by \cite{R3}, where theorems on reconstruction with infinite-dimensional
representations were obtained with the use of a certain generalization
of rigid categories (called reflexive monoidal categories), though
the recognition problem remained unconsidered. Among the huge number
of works, that followed, on reconstruction using rigid categories,
we would like to mention \cite{PAR2} (and \cite{PAR1}), which seems
to provide the most general context for reconstruction in monoidal
categories (we do not consider extensions to higher algebraic structures)
and on which we build heavily. This list is by no means exhaustive
and is very subjective.

In reconstruction theory, one can reconstruct the group $G$ from
the forgetful functor $\funct F:Rep\left(G\right)\to{\rm Vect}_{K}^{fd}$
to the category of finite-dimensional vector spaces as the group of
natural monoidal automorphisms of $\funct F$ (in the case of algebraic
groups, this is a simple consequence of the Yoneda lemma). For the
reconstruction of a Hopf algebra (of functions on $G$) one reconstruct
it as a predual object, the coendomorphism object $\coendf{\funct F}{}$.
Since the coalgebra theory is simpler than the theory of algebras,
the second approach turn out to be easier and is followed by most
authors for the reconstruction of coalgebras, bialgebras and Hopf
algebras. In a general rigid monoidal category $\mc C_{0}$ one can
construct $\coendf{\funct F}{}$ for a functor $\funct F:\mc B\to\mc C_{0}$
as a colimit of a diagram of objects of the form $\funct F\left(B\right)^{\vee}\Cten C\funct F\left(B\right)$,
which requires us to assume the existence of a cocomplete category
$\mc C\supset\mc C_{0}$ and to assume either $\mc B$ or $\mc C_{0}$
to be small for the above mentioned colimit to exist. These conditions
require us to consider fiber functors into categories of objects that
satisfy some finiteness conditions (finite-dimensional vector spaces,
finitely generated projective modules, etc.). One can also define
$\coendf{\funct F}{}$ as a representing object for the functor ${\rm Nat}\left(\funct F,\funct F\Cten C-\right):\mc C\to{\rm Sets}$
of monoidal natural transformations, assuming that $\mc C$ is locally
small (otherwise ${\rm Nat}\left(\funct F,\funct F\Cten C-\right)$
is not a functor to ${\rm Sets}$). While this definition works in
general, to prove the existence of $\coendf{\funct F}{}$ one still
need a direct construction of it and thus is bound to the case of
functors $\funct F:\mc B\to\mc C_{0}$ into a rigid category.

The starting point (and the essence) for this paper is the observation
that for the computation of the $\coendf{\funct F}{}$ of the fiber
functor $\funct F$ one does not need the category $\mc C_{0}$ to
be rigid, but rather it only has to be coclosed, i.e. the tensor product
functor $Y\Cten C-$ has to have a left adjoint functor $\cohom[C]-Y$.
According to the theorem of Eilenberg and Watts, any such functor
for categories of modules ${\rm Mod}_{R}$ over a ring $R$ is itself
a tensor product functor $Y^{\vee}\Cten C-$ and one can take $Y^{\vee}$
as a dual of $Y$, recovering the rigid structure from the adjuction.
Thus in the case of categories of modules the coclosedness requirement
is equivalent to requiring rigidity. There is a conjecture that Eilenberg-Watts
theorem holds in any abelian monoidal category, so it is no wonder
that the fact that left adjoints can be used in Tannaka reconstruction
was either not noticed or not given importance by researchers. Although
the author is not aware of the status of this conjecture for abelian
categories, in nonabelian case it is false (and this paper contains
a counter-example). Thus the use of cohoms allows us to compute $\coendf{\funct F}{}$
in categories that are not rigid, i.e. without usual finiteness assumptions
on objects of the category $\mc C_{0}$, and thus one can do reconstruction
in a more general setting. Another feature of this work is that we
never require preservation of arbitrary colimits by tensor products,
which is common for all previous works. While in algebra this requirement
is usually not restrictive (since algebraic tensor product functor
is a left adjoint), topological tensor products seldom enjoy this
property and thus in the topological setting one needs reconstruction
theorems that do not require preservation of arbitrary colimits.

Let us briefly outline the content of the paper.

In section 1 we review the definition and some basic properties of
cohomomorphisms. We define $\coendf{\funct F}{}$ in coclosed categories
and work out its basic properties by giving direct proofs, without
using ${\rm Nat}\left(\funct F,\funct F\Cten C-\right)$. We think
it is useful for understanding how these properties work, and also
it allows to apply these constructions to the case when the category
$\mc C$ is not locally small (though in our applications it is locally
small). We use the framework of $\mc C$-categories, introduced by
B. Pareigis for reconstruction problems, which we briefly review in
preliminaries, along with some results of \cite{PAR2}. We also introduce
the concept of $\mc C$-cowedge, corresponding to $\mc C$-natural
transformations (\cite{PAR2}) and give direct construction of $\coendf{\funct F}{\mc C}$.
This allows us to state reconstruction and recognition theorems as
a consequence of the results of \cite{PAR2}, though we slightly relax
the assumptions, to adapt them for our applications.

In section 2 we apply our results to state reconstruction and recognition
theorems in the categories of topological vector spaces over nonarchimedean
field $K$, which was our original motivation to study this question.
The category $\underline{LS}$ of LS-spaces, that we consider, is
the category of spaces of locally analytic representation theory,
and thus our results can be viewed as (a part of) Tannaka-Krein duality
for compact $p$-adic groups. We also prove reconstruction theorems
for a bialgebra of compact type from its category of Banach comodules.

Most of our results are formulated in their simplest versions, but
the construction is very flexible and can be adapted to different
settings, in which one might need to do reconstruction. We show this
in section 3, where we adapt out results for reconstruction in the
category of Banach spaces $\Ban$. The usual problem with $\Ban$
is that it is not cocomplete (and not complete), thus one cannot take
a colimit to construct $\coendf F{}$. We overcome this problem by
putting a restriction on cowedges and natural transformations that
we call boundedness (one will see that this name is natural) and show
that the results of section 1 survive this change. An application of this result can be reconstruction for rigid analytic groups of split semi-simple simply connected group schemes, since in this case all finite-dimensional comodules are algebraic and will satisfy all conditions of the construction.

\section*{Preliminaries}

Throughout this paper $\left(\cal C,\Cten C\right)$ is a monoidal
category.

\subsection{$\cal C$-categories. }

We briefly review some notions and results from \cite{PAR2}. For
any unknown term with no reference one should look there for explanations.
\begin{defn}
Let $\left(\cal C,\Cten C\right)$ be a monoidal category.
\begin{itemize}
\item A category $\cal B$ together with a bifunctor $\Cten{CB}:\cal C\times\cal B\to\cal B$
and coherent natural isomorphisms $\beta:\left(X\Cten CY\right)\Cten{CB}P\to X\Cten{CB}\left(Y\Cten{CB}P\right)$
(for $X,Y\in\cal C$, $P\in\cal B$) and $\pi:I\Cten{CB}P\to P$ will
be called a (left) $\cal C$-\emph{category}. $\cal C$ is called
\emph{control category}.
\item Let $\cal B$ and $\cal B'$ be ${\cal C}$-categories. A functor
$\funct F:\cal B\to\cal B'$ together with a coherent natural isomorphism
$\xi:\funct F\left(X\Cten{CB}P\right)\to X\Cten{CB{\rm '}}\funct F\left(P\right)$
is called a $\cal C$-\emph{functor}. Note that if $\cal B'$ is also
a monoidal category and $M\in\cal B'$ then $\funct F\Cten{B{\rm '}}M$
is also a $\cal C$-functor.
\item Let $\funct F:\cal B\to\cal B'$ and $\funct F':\cal B\to\cal B'$
be $\cal C$-functors. A $\cal C$-natural transformation (or a $\cal C$-morphism)
$\phi:\funct F\overset{\centerdot}{\to}\funct F'$ is a natural transformation,
compatible with the natural isomorphisms $\xi$ and $\xi'$. The collection
of $\cal C$-natural transformations from $\funct F$ to $\funct F'$
will be denoted by $\nat{\cal C}{\funct F}{\funct F'}$. The set of
all natural transformations $\nat{}{\funct F}{\funct F'}$ corresponds
to the one-element control category.
\end{itemize}
\end{defn}
\smallskip{}

\begin{defn}
Let $\cal C$ be a braided monoidal category.
\begin{itemize}
\item Let $\cal B,$ $\cal B'$ and $\cal B''$ be $\cal C$-categories.
A bifunctor $\cal B\times\cal B'\to\cal B''$ is called $\cal C$\emph{-bifunctor},
if it is a $\cal C$-functor on each variable such that bifunctor
and $\cal C$-functor structures are compatible (in the sense of \cite[2.4]{PAR2}).
Similar to $\cal C$-morphisms, one defines $\cal C$\emph{-bimorphisms}
of $\cal C$-bifunctors. 
\item Let $\cal A$ be a $\cal C$-category with coherent structure of a
monoidal category with tensor product $\Cten A$. If $\Cten A$ is
a coherent $\cal C$-bifunctor, then $\cal A$ is called $\cal C$-\emph{monoidal
category}. It's monoidal structure morphisms then are $\cal C$-morphisms.
\item A monoidal functor between two $\cal C$-monoidal categories is called
$\cal C$\emph{-monoidal} if it is a $\cal C$-functor and it's monoidal
functor structure morphism is a $\cal C$-bimorphism. 
\item A natural transformation is called $\cal C$\emph{-monoidal} if it
is monoidal and $\cal C$-natural. 
\item A $\cal C$-monoidal category $\cal A$ is called $\cal C$\emph{-braided}
if the braidings in $\cal A$ and $\cal C$ are coherent (but braiding
on $\cal A$ is not a $\cal C$-morphism).
\item An object $P$ in a ${\cal C}$-monoidal category $\cal A$ is called
${\cal C}$\emph{-central} \cite[2.9]{PAR2} if 
\[
X\Cten{CA}P\Cten AQ\to P\Cten{CA}X\Cten AQ\to X\Cten{CA}P\Cten AQ=\id{X\Cten{CA}P\Cten AQ}
\]
for all $X\in\cal C$ and $Q\in\cal A$.
\end{itemize}
\end{defn}
\begin{prop}
\cite[2.10]{PAR2} Let $\cal A$ be $\cal C$-braided $\cal C$-monoidal
category. Let B be a $\cal C$-central bialgebra in $\cal A$, $C$
be a coalgebra in $\cal A$ and $z:C\to B$ be a coalgebra morphism.
Then 
\begin{enumerate}
\item $\comodc A-C$ is a $\cal C$-category
\item $\comodc A-B$ is a $\cal C$-monoidal category
\item $\funct F^{z}:\comodc A-C\to\comodc A-B$ is a $\cal C$-functor
\item the forgetful functor $\funct F:\comodc A-C\to\cal A$ is a $\cal C$-functor
\item if $C$ is a $\cal C$-central bialgebra and $z:C\to B$ is a bialgebra
morphism then $\funct F^{z}$ is a $\cal C$-monoidal functor
\item the forgetful functor $\funct F:\comodc A-B\to\cal A$ is a $\cal C$-monoidal
functor
\end{enumerate}
\end{prop}
Recall the (slightly changed) notion of $\cal C_{0}$-generated coalgebra
\cite[2.6]{PAR2}.
\begin{defn}
\label{C0gendefn}Let $\cal C$ be a monoidal category, $\cal C_{0}\subset\cal C$
be a full monoidal subcategory and $I$ be a poset. The coalgebra
$C\in\cal C$ is called (\emph{right}) $\cal C_{0}-I$\emph{-generated}
if the following holds:
\begin{enumerate}
\item $C$ is a colimit in $\cal C$ of an $I$-diagram of objects $C_{i}\in\cal C_{0}$;
\item all morphisms $\id X\Cten Cj_{i}\Cten C\id M:X\Cten CC_{i}\Cten CM\to X\Cten CC\Cten CM$
are monomorphisms in $\cal C$, where $X\in\cal C_{0}$, $M\in\cal C$
and $j_{i}:C_{i}\to C$ are the monomorphisms from the colimit diagram;
\item every $C_{i}$ is a subcoalgebra of $C$ via $j_{i}:C_{i}\to C$;
\item If $\left(P,\rho_{P}\right)$ is a (right) comodule over $C$ and
$P\in\cal C_{0}$ then $\exists$ $i$ and $\rho_{P,i}:P\to P\Cten CC_{i}$
such that $\rho_{p}=\left(\id P\Cten Cj_{i}\right)\circ\rho_{P,i}$.
\end{enumerate}
If we don't want to specify $I$ or it's choice is clear we say that
$C$ is $\cal C_{0}$\emph{-generated.}\end{defn}
\begin{rem}
For a $\cal C_{0}-I$-generated\emph{ }coalgebra $C\in\mc C$ and
any subcoalgebra $C'\hookrightarrow C$ such that $C'\in\mc C_{0}$
the identity $\id C=\left(\epsilon_{C}\bten\id C\right)\circ\cm_{C}$
and \ref{C0gendefn}.4 imply that we have a monomorphism $C'\hookrightarrow C_{i}$
for some $i\in I$.\end{rem}
\begin{example}
Fundamental theorem of coalgebras: for any field $K$ every coassociative
coalgebra in the category $Vect_{K}$ of all vector spaces is $Vect_{K}^{fd}$-generated,
where $Vect_{K}^{fd}$ is the category of all finite-dimensional vector
spaces.
\end{example}
\smallskip{}

\begin{example}
Let $K$ be a locally compact nonarchimedean field and $G$ be a compact
locally $K$-analytic group. Then the coalgebra $C^{la}\left(G,K\right)$
of locally analytic functions on $G$ is $\Ban-\mathbb{N}$-generated
in the category $\underline{LCTVS}$ of locally convex topological
vector spaces or in the category $\underline{LB}$ of LB-spaces. 
\end{example}

\subsection{\label{sub:Coendomorphism-of-bifunctors}Coendomorphism of bifunctors
$\cal C\times\cal C^{op}\to\cal D$.}

Let $\funct F:\cal C\times\cal C^{op}\to\cal D$ be a bifunctor. 
\begin{defn}
A \emph{dinatural transformation} $\mu:\funct F\ddot{\to}d$ from
$\funct F$ to a constant bifunctor with value $d\in{\rm Ob}\left(\cal D\right)$
is a family of morphisms $\mu_{c}:\funct F\left(c,c\right)\to d$
such that for any morphism $f:c\to c'$ the following diagram commutes\begin{center}
\begin{tikzpicture}[node distance=2cm, auto]    
\node (fc) {$\funct F\left(c,c'\right)$};
\node (fcc') [below of=fc] {$\funct F\left(c',c'\right)$};
\node (fcc) [right= and 3cm of fc] {$\funct F\left(c,c\right)$};   
\node (d) [below of=fcc] {$d$};   
\draw[->] (fc) to node {$\funct F\left(\id{c},f^{op}\right)$} (fcc);
\draw[->] (fc) to node {$\funct F\left(f,\id{c'}\right)$} (fcc'); 
\draw[->] (fcc) to node {$\mu_c$} (d); 
\draw[->] (fcc') to node [swap] {$\mu_{c'}$} (d); 
\end{tikzpicture} 
\end{center}Similarly one defines a dinatural transformation from a constant to
$\funct F$.
\end{defn}
We will also call dinatural transformation from a constant to $\funct F$
a \emph{wedge} over $\funct F$. A dinatural transformation from $\funct F$
to a constant will be a \emph{cowedge} over $\funct F$. The definition
of a morphism of (co)wedges is clear. (Co)wedges over $\funct F$
form a category.
\begin{defn}
\label{end-coend-defn}${\rm end}\left(\funct F\right)$ is a terminal
object in the category of wedges over $\funct F$. $\coendf{\funct F}{}$
is an initial object in the category of cowedges over $\funct F$.
\end{defn}

\subsection{\label{sub:subdivision-category,-existence}subdivision category,
existence of coend. \cite[IX.5]{MAC}}

To a category $\cal C$ we associate its subdivision category $\cal C^{\S}$
in the following way: we set ${\rm Ob}\left(\cal C^{\S}\right)={\rm Ob\left(\cal C\right)}\cup{\rm Mor}\left(\cal C\right)$
(note that $c\in\cal C$ and $\id c$ give different objects $c^{\S}$
and $\id c^{\S}$ in $\cal C^{\S}$). The morphisms are the identity
morphisms for the above mentioned objects and also if $f:c\to d$
is the morphism in $\cal C$ then it gives rise to the two morphisms
$c^{\S}\to f^{\S}$ and $f^{\S}\to d^{\S}$ in $\cal C^{\S}$.

For each bifunctor $\funct F:\cal C\times\cal C^{op}\to\cal D$ we
define a functor $\funct F^{\S}:\cal C^{\S}\to\cal D$ via the diagram
\begin{center}
\begin{tikzpicture}[node distance=3cm, auto]    
\node (c) {$c^\S$};
\node (f) [right of=c] {$f^\S$};
\node (d) [right of=f] {$d^\S$};
\node (Fc) [below= and 1cm of c] {$\funct F\left(c,c\right)$};
\node (Ff) [right of=Fc] {$\funct F\left(d,c\right)$};
\node (Fd) [right of=Ff] {$\funct F\left(d,d\right)$};
\draw[->] (Fd) to node [swap] {$\funct F\left(\id{d},f^{op}\right)$} (Ff);
\draw[->] (Fc) to node {$\funct F\left(f,\id{c}\right)$} (Ff); 
\draw[->] (c) to node {} (f);
\draw[->] (d) to node {} (f);
\draw[->] (c) to node {} (Fc);
\draw[->] (f) to node {} (Ff);
\draw[->] (d) to node {} (Fd);
\end{tikzpicture} 
\end{center}A (co)cone over $\funct F^{\S}$ is exactly a (co)wedge over $\funct F$
and vice verse. Thus a (co)limit of $\funct F^{\S}$ exists if and
only if a (co)end of $\funct F$ exists (\cite[IX.5.1]{MAC}) and
\[
\lim_{\to}\funct F^{\S}={\rm end}\,\funct F,\quad\lim_{\from}\funct F^{\S}=\coend{}{\funct F}.
\]
Thus if $\cal C$ is small and $\cal D$ is complete (cocomplete)
then ${\rm end}\,\funct F$ ($\coend{}{\funct F}$) exists (\cite[IX.5.2]{MAC}).

If $\cal C$ or $\cal D$ is small, one can also compute $\coendf{\funct F}{}$
as the coequalizer 
\[
\coprod_{f\in{\rm Mor}\left(C\right)}\funct F\left({\rm dom}\left(f\right),{\rm codom}\left(f\right)\right)\underset{q}{\overset{p}{\rightrightarrows}}\coprod_{c\in{\rm Ob}\left(C\right)}\funct F\left(c,c\right)\to\coendf{\funct F}{},
\]
where the equalized pair is defined by equations 
\[
p\circ i_{f}=i_{{\rm dom}\left(f\right)}\circ\funct F\left(\id{{\rm dom}\left(f\right)},f^{op}\right),
\]
\[
q\circ i_{f}=i_{{\rm codom}\left(f\right)}\circ\funct F\left(f,\id{{\rm codom}\left(f\right)}\right).
\]

\subsection{Functor $\nat{\cal C}{\funct F}{\funct F\ten-}$.}

Let $\cal A$, $\cal B$ be $\cal C$-categories and $\funct F:\cal B\to\cal A$
be a $\cal C$-functor. If $\cal B$ is small and $\cal A$ is locally
small and $\cal C$-monoidal, then the $\cal C$-natural transformations
form a functor 
\[
\nat{\cal C}{\funct F}{\funct F\Cten A-}:\cal A\to{\rm Sets}.
\]
Suppose $\nat{\cal C}{\funct F}{\funct F\Cten A-}$ is representable.
By analogy with the case of all natural transformations, we denote
the representing object by $\coend C{\funct F}$ with the universal
$\mc C$-morphism $\delta:\funct F\to\funct F\Cten A\coend C{\funct F}$.

$\funct F^{2}:=\funct F\Cten A\funct F:\mc B\times\mc B\to\mc A$
is $\mc C$-bifunctor, and so is $\funct F^{2}\Cten AM$ for all $M\in\mc A$.
The $\cal C$-natural transformations $\nat{\cal C}{\funct F^{2}}{\funct F^{2}\Cten AM}$
form a set and thus define a functor 
\[
\nat{\cal C}{\funct F^{2}}{\funct F^{2}\Cten A-}:\cal A\to{\rm Set}.
\]
If $\coend C{\funct F}$ is $\mc C$-central (always true if $\mc C=\mc A$
and $\mc A$ is symmetric) then 
\[
\delta_{2}=\left(\id{\funct F}\Cten A\tau\Cten A\id{\coend C{\funct F}}\right)\circ\left(\delta\Cten A\delta\right):\funct F^{2}\to\funct F^{2}\Cten A\coend C{\funct F}^{2}
\]
 is a $\mc C$-bimorphism. Similar statement holds for $\nat{\cal C}{\funct F^{n}}{\funct F^{n}\Cten A-}$
and $\delta_{n}$. 
\begin{defn}
\cite[3.5]{PAR2} We will call $\nat{\cal C}{\funct F}{\funct F\Cten A-}$
$n$\emph{-representable} if $\coend C{\funct F}$ is $\mc C$-central
and $\nat{\cal C}{\funct F^{n}}{\funct F^{n}\Cten A-}$ is representable
with $\coend C{\funct F}^{n}$ and $\delta_{n}$, and \emph{multirepresentable}
if it is $n$-representable for all $n$.
\end{defn}
In the following propositions we summarize some results from \cite{PAR2}.
\begin{prop}
\label{prop:The-functor-Nat-properties}The functor $\nat{\cal C}{\funct F}{\funct F\Cten A-}$
has the following properties (\cite[3.3, 3.6]{PAR2}):
\begin{itemize}
\item $\nat{\cal C}{\funct F}{\funct F\Cten AM}\cong\cal A\left(\coend C{\funct F},M\right)$;
\item (equivalent to representability) for every $\mc C$-morphism $\phi:\funct F\to\funct F\Cten AM$
there exists a morphism $\psi:\coend C{\funct F}\to M$ such that
the diagram\begin{center}
\begin{tikzpicture}[node distance=2cm, auto]    
\node (x) {$\funct F\left(X\right)$};   
\node (xc) [right= and 2cm of x] {$\funct F\left(X\right)\Cten A \coend C{\funct F} $};   
\node (xm) [below of=xc] {$\funct F\left(X\right)\Cten A M$};   
\draw[->] (x) to node {$\delta_X$} (xc); 
\draw[->] (x) to node [swap] {$\phi_X$} (xm); 
\draw[->] (xc) to node {$\id{\funct F\left(X\right)}\Cten A \psi$} (xm); 
\end{tikzpicture} 
\end{center}commutes for every $X\in\mc B$;
\item $\coend C{\funct F}$ is a coalgebra, unique up to isomorphism of
coalgebras;{*}
\item $\forall P\in\mc B$: $\funct F\left(P\right)$ is a $\coend C{\funct F}$-comodule,
every $\funct F\left(f\right)$ is a comodule morphism;
\item $\forall P\in\mc B,\!\forall X\in\mc C$: $\funct F\left(X\Cten{CB}P\right)$
is isomorphic to $X\Cten{CA}\funct F\left(P\right)$ as $\coend C{\funct F}$-comodule;
\end{itemize}
Now let $\cal A$ be ${\cal C}$-braided ${\cal C}$-monoidal, ${\cal B}$
be ${\cal C}$-monoidal and $\funct F:\cal B\to\cal A$ be ${\cal C}$-monoidal
functor. 
\begin{itemize}
\item if $\nat{\cal C}{\funct F}{\funct F\Cten A-}$ is multirepresentable
then $\coend C{\funct F}$ is a bialgebra in $\mc A$. It is unique
up to isomorphism of bialgebras;
\item if $\funct F$ factors through the full subcategory $\mc A_{0}$ of
rigid objects of $\mc A$, then $\coend C{\funct F}$ is Hopf algebra
in $\mc A$;
\item if $\mc C\subseteq\mc C'$ then there is a coalgebra epimorphism $\coend C{\funct F}\twoheadrightarrow\coendf{\funct F}{\mc C'}$
\cite[5.2]{PAR2}. In particular, there exists an epimorphism $\coend{}{\funct F}\twoheadrightarrow\coendf{\funct F}{\mc C}$.
\end{itemize}
\end{prop}
\smallskip{}

\begin{prop}
(Reconstruction theorem) Let $\cal C$ be a braided monoidal category
and $\cal C_{0}$ be a full braided monoidal subcategory.
\begin{itemize}
\item Let $C$ be a $\cal C_{0}$-generated coalgebra in $\cal C$ and $\funct F:\comod{\cal C_{0}}-C\to\cal C_{0}\subset\cal C$
be the forgetful functor. Then the functor $\nat{\cal C_{0}}{\funct F}{\funct F\Cten C-}$
is representable by $C$ \cite[4.3]{PAR2};
\item Let $\cal C$ also be cocomplete and let $\ten$ preserve colimits
in both variables. Let $C$ also be a $\cal C_{0}$-central bialgebra.
Then $\nat{\cal C_{0}}{\funct F}{\funct F\Cten C-}$ is multirepresentable
by $C$ \cite[4.5]{PAR2}.
\end{itemize}
\end{prop}
\smallskip{}

\begin{prop}
(Recognition theorem \cite[4.7]{PAR2}) Let $\cal C$ be a cocomplete
braided monoidal category and $\cal C_{0}$ be a (locally small) full
braided monoidal subcategory of rigid objects. Assume that $\Cten C$
preserve colimits in both variables. Let $\cal B$ be a small category
and $\funct F:\cal B\to\cal C_{0}\subset\cal C$ be a functor. 
\begin{itemize}
\item the functor $\nat{}{\funct F}{\funct F\Cten C-}$ is multirepresentable;
\item if $\cal B$ be a $\cal C_{0}$-category and $\funct F:\cal B\to\cal C_{0}$
be a $\cal C_{0}$-functor, then $\nat{\cal C_{0}}{\funct F}{\funct F\Cten C-}$
is representable;
\item if $\coendf{\funct F}{\cal C_{0}}$ is $\cal C_{0}$-central then
$\nat{\cal C_{0}}{\funct F}{\funct F\Cten C-}$ is multirepresentable.
\end{itemize}
\end{prop}

\section{Coclosed categories.}

\subsection{Cohomomorphisms.}
\begin{defn}
Let $\cal C$ be a monoidal category and $X,Y\in\cal C$. And object
$\cohom[C]XY$ (or simply $\cohom XY$) (with the morphism $\coev{X,Y}:X\to Y\Cten C\cohom[C]XY$)
is called (right) \emph{cohomomorphism} object for $X$ and $Y$ if
for every $Z\in\cal C$ and every morphism $\phi:X\to Y\Cten CZ$
there is a unique morphism ${\rm coact}_{X,Y,Z}\left(\phi\right):\cohom[C]XY\to Z$
satisfying the following commutative diagram\begin{center}
\begin{tikzpicture}[node distance=2cm, auto]    
\node (x) {$X$};   
\node (yc) [right= and 2cm of x] {$Y\Cten C \cohom[C]XY $};   
\node (yz) [above of=yc] {$Y\Cten C Z$};   
\draw[->] (x) to node {$\phi$} (yz); 
\draw[->] (x) to node [swap] {$\coev {X,Y}$} (yc); 
\draw[->] (yc) to node [swap] {$\mathrm{id}_Y\Cten C {\rm coact}_{X,Y,Z}\left(\phi\right)$} (yz); 
\end{tikzpicture} 
\end{center}
\end{defn}
If it exists, $\cohom[C]XY$ is unique up to an isomorphism.

Instead of ${\rm coact}_{X,Y,Z}\left(\phi\right)$ we might write
just ${\rm coact}\left(\phi\right)$ or $\widehat{\phi}$.
\begin{defn}
We say that the category $\cal C$ is \emph{right coclosed} if for
all $X,Y\in\cal C$ there exists a right cohomomorphism object. We
also say that a subcategory $\mc C_{0}\subset\mc C$ is \emph{right
coclosed in }$\mc C$ if $\cohom[C]XY\in\mc C$ exists for any $X,Y\in\mc C_{0}$. 
\end{defn}
Similarly one can define left cohomomorphism objects and left coclosed
categories. In this paper we will always consider right coclosed categories,
unless we explicitly mention the opposite.

In a (right) coclosed category $\cal C$ a map $\phi:X\to Z\Cten CC$
induce the map 
\begin{equation}
\cm\equiv{\rm coact}\left(\left(\coev{Z,Y}\Cten C\id{\cohom[C]XZ}\right)\circ\phi\right):\label{eq:indcoact}
\end{equation}
\[
\cohom[C]XY\to\cohom[C]ZY\Cten CC
\]
 via diagram 

\begin{center}
\begin{tikzpicture}[node distance=2cm, auto]    
\node (x) {$X$};   
\node (zc) [right= and 4.5cm of x] {$Z\Cten C C$};   
\node (ycc) [below of =zc] {$Y\Cten C\cohom[C]ZY\Cten C C$};
\node (yc) [below of =x] {$Y\Cten C\cohom[C]XY$};
\draw[->] (x) to node {$\phi$} (zc); 
\draw[->] (x) to node [swap] {$\coev {X,Y}$} (yc); 
\draw[->] (yc) to node {$\id{Y}\Cten C\Delta$} (ycc);
\draw[->] (zc) to node {$\coev {Z,Y}\Cten C\id{\cohom[C]XZ}$} (ycc); 
\end{tikzpicture} 
\end{center}

If we take $\phi:=\coev{X,Z}:X\to Z\ten\cohom[C]XZ$ it gives us cocomposition
map $\Delta_{X,Y,Z}:\cohom[C]XY\to\cohom[C]ZY\Cten C\cohom[C]XZ$ 

One can see that $\cohom[C]--$ defines a bifunctor ${\rm Cohom}:\cal C\times\cal C^{op}\to\cal C$
. The functor ${\rm Cohom}\left(-,Y\right)$ is the left adjoint to
the functor $Y\Cten C-$, i.e. we have an isomorphism 
\[
\cal C\left(\cohom[C]XY,Z\right)\cong\cal C\left(X,Y\Cten CZ\right).
\]

\begin{lem}
For any object $T\in\cal C$ we have the following identities 
\begin{enumerate}
\item $\cal C\left(\cohom[C]{\cohom[C]XY}Z,T\right)=\cal C\left(\cohom[C]XY,Z\Cten CT\right)=$\\$=\cal C\left(X,Y\Cten CZ\Cten CT\right)=\cal C\left(\cohom[C]X{Y\Cten CZ},T\right);$
\item $\cal C\left(\cohom[C]XY,I\right)\cong\cal C\left(X,Y\Cten CI\right)\cong\cal C\left(X,Y\right);$
\item $\cohom[C]XI\cong X\mbox{ with }\coev{X,I}=\id X$;
\end{enumerate}
\end{lem}
\begin{example}
Let $\mc C={\rm vec}_{K}$ be the category of finite-dimensional vector
spaces over the field $K$ with the tensor structure given by the
tensor product of vector spaces. For $X,Y\in{\rm vec}$ lets compute
$\cohomm[{\rm vec}_{K}]XY$ (without using hom-tensor adjunction). 

Let $\left\{ x_{i}\right\} _{i=1,...,n}$ and $\left\{ y_{j}\right\} _{j=1,...,m}$
be the bases of $X$ and $Y$ correspondingly. Let $Y^{*}$ is the
linear dual space of $Y$ and $\left\{ y_{j}'\right\} _{j=1,...,m}$
be its basis, dual to $\left\{ y_{j}\right\} _{j=1,...,m}$. For any
$K$-linear map $\phi:X\to Y\ten Z_{\phi}$ with $Z_{\phi}\in{\rm vec}_{K}$
denote by $z_{ij}^{\phi}=\left(y_{j}'\bten\id{Z_{\phi}}\right)\circ\phi\left(x_{i}\right)$.
$\phi$ is defined via it's values on $x_{i}$, which can be written
as $\phi\left(x_{i}\right)=\sum y_{j}\ten z_{ij}^{\phi}$. Without
loss of generality we can say that $\left\{ z_{ij}^{\phi}\right\} $
span the whole $Z_{\phi}$. The largest such $Z_{\phi}$ one can get
is $Z_{\Phi}:=K^{nm}$, i.e. when all $z_{ij}^{\phi}$ are linearly
independent and all other cases are quotients $\pi_{\phi}:Z_{\Phi}\twoheadrightarrow Z_{\phi}$
by the corresponding relations $R_{\phi}:=\ker\left(\pi_{\phi}\right)$
between $z_{ij}^{\phi}$. For any two maps $\phi_{n}:X\to Y\ten Z_{\phi_{n}}$,
$n=1,2$, the transition map $\pi_{1,2}:Z_{\phi_{1}}\to Z_{\phi_{2}}$
such that $\phi_{2}=\left(\id Y\ten\pi_{1,2}\right)\circ\phi_{1}$,
if exists, is unique, since it defined by the values $z_{ij}^{\phi_{k}}$.
Thus we have a diagram $\left\{ Z_{\phi}\right\} $ with limit $Z_{\Phi}$.
Since this diagram is finite, it is preserved by the tensor product
and $Y\ten Z_{\Phi}=\lim\left(Y\ten Z_{\phi}\right)$. Thus we get
the map $\Phi:X\to Y\ten Z_{\Phi}$.

This means that $\cohomm[{\rm vec}]XY\cong K^{nm}$ and $\coev{X,Y}=\Phi$.
Since ${\rm cohom}$ is covariant in $X$ and contravariant in $Y$,
we write it as $\cohomm[{\rm vec}_{K}]XY\cong Y^{*}\ten X$.
\end{example}
\smallskip{}

\begin{rem}
The above argument shows that in $\mc C={\rm Vec}_{K}$ for finite-dimensional
$X$ and $Y$ we have $\cohomm[{\rm Vec}_{K}]XY\cong Y^{*}\ten X$.
It also explains why, if either $X$ or $Y$ has infinite dimension,
$\cohomm[{\rm Vec}_{K}]XY$ does not exist, since in this case the
diagram $\left\{ Z_{\phi}\right\} $ is (filtered) infinite and its
limit is not preserved by tensoring with $Y$.
\end{rem}

\subsection{Coendomorphisms.}
\begin{defn}
We define the coendomorphism object of $X\in\cal C$ as $\coend CX:=\cohom[C]XX$.
We will also write $\coev X$ for $\coev{X,X}$.\end{defn}
\begin{lem}
Let $X\in\cal C$ and suppose $\coend CX$ exists. Then
\begin{enumerate}
\item $\coend CX$ is a coalgebra in $\cal C$ with the comultiplication
\[
\Delta_{\coend CX}:={\rm coact}\left(\left(\coev X\Cten C\id{\coend CX}\right)\circ\coev X\right)
\]
and the counit $\epsilon_{\coend CX}:={\rm coact}\left(\id X\right)$;
\item $X$ is a right comodule over $\coend CX$ via $\rho_{X}:=\coev X$.
\end{enumerate}
\end{lem}
\begin{proof}
The check of coalgebra axioms for $\left(\coend CX,\Delta_{\coend CX},\epsilon_{\coend CX}\right)$
is straightforward. 

The check of coassociativity for $\rho_{X}$ is also straightforward.
From the relation $\id X=\left(\id X\Cten C\epsilon_{\coend CX}\right)\circ\coev X$
we get that $X$ is a comodule over $\coend CX$. \end{proof}
\begin{rem}
Let $\phi:X\to X\Cten CC$ be a morphism in $\cal C$. Then 
\begin{itemize}
\item one has a morphism 
\[
\rho_{\phi}:={\rm coact}\left(\left(\coev X\Cten C\id C\right)\circ\phi\right):\coend CX\to\coend CX\Cten CC;
\]

\item if $C$ is a coalgebra and $\phi$ is a comodule coaction, then 

\begin{itemize}
\item $\rho_{\phi}$ is a comodule coaction;
\item ${\rm coact}\left(\phi\right):\coend CX\to C$ is a coalgebra morphism;
\end{itemize}
\item one can reconstruct the coalgebra morphism ${\rm coact}\left(\rho_{\phi}\right):\coend CX\to C$
as 
\[
{\rm coact}\left(\rho_{\phi}\right)=\left(\epsilon_{\coend CX}\Cten C\id C\right)\circ\rho_{\phi}.
\]

\end{itemize}
\end{rem}

\subsection{Rigid tensor categories}

In the treatment of rigid monoidal categories we will follow \cite{PAR1}.
\begin{defn}
Let $X\in\cal C$. $\left(X^{*},ev:X^{*}\Cten CX\to I\right)$ is
a (left) dual of $X$ if there exists $db:I\to X\Cten CX^{*}$, s.t.
\[
\left(\id X\Cten Cev\right)\circ\left(db\Cten C\id X\right)=\id X
\]
 and 
\[
\left(ev\Cten C\id X\right)\circ\left(\id X\Cten Cdb\right)=\id X.
\]
$\cal C$ is (left) \emph{rigid} if every object has a (left) dual.
\end{defn}
Dual objects have the following properties:
\begin{itemize}
\item $\left(X^{*},ev\right)$ is a left dual for $X$ iff $-\Cten CX:\cal C\to\cal C$
is a left adjoint to $-\Cten CX^{*}$, i.e. 
\[
\cal C\left(-\Cten CX,-\right)\cong\cal C\left(-,-\Cten CX^{*}\right)
\]
 iff $\mbox{}$
\[
X^{*}\Cten C-:\cal C\to\cal C\mbox{ is a left adjoint to }X\Cten C-
\]
\cite[3.3.3-5]{PAR1};
\item For all $X$, $X^{*}$ and $db$ are unique; 
\item $X\Cten CX^{*}$ is an algebra \cite[3.3.14]{PAR1};
\item $X^{*}\Cten CX$ is a coalgebra \cite[3.3.14]{PAR1};
\item For all $X,Y\in\cal C$ $Y^{*}\Cten CX\cong\cohom[C]XY$ \cite[3.3.5]{PAR1},
i.e. left rigid categories are right coclosed;
\item duality operation forms a functor $\left(-\right)^{*}:\cal C\to\cal C^{op}$
(follows from \cite[3.3.8]{PAR1}).
\end{itemize}
From the above properties it is clear that in a braided rigid category
${\rm Cohom}$ is a tensor bifunctor, i.e. 
\[
\cohom[C]{X\Cten CU}{Y\Cten CV}\cong\left(Y\Cten CV\right)^{*}\Cten C\left(X\Cten CU\right)\cong V^{*}\Cten CY^{*}\Cten CX\Cten CU\cong
\]
\[
\cong Y^{*}\Cten CX\Cten CV^{*}\Cten CU\cong\cohom[C]XY\Cten C\cohom[C]UV.
\]

\begin{lem}
Let $\mc C$ be a braided right coclosed monoidal category such that
${\rm Cohom}$ is a tensor bifunctor. Then $\mc C$ is a left rigid
category.\end{lem}
\begin{proof}
We have the following isomorphism 
\[
\cohom[C]XY\cong\cohom[C]{X\Cten CI}{I\Cten CY}\cong
\]
\[
\cohom[C]XI\Cten C\cohom[C]IY\cong\cohom[C]IY\Cten CX.
\]
By the definition of $\cohom[C]XY$, the functor $\cohom[C]IY\Cten C-$
is the left adjoint to the functor $Y\Cten C-$. By \cite[3.3.5]{PAR1}
$\cohom[C]IY$ is the left dual of $Y$.
\end{proof}

\subsection{Coendomorphism coalgebra of a functor.}

In this subsection let $\cal C$ be a monoidal category, $\mc C_{0}\subset C$
be a subcategory, coclosed in $\mc C$, and $\funct F:\cal D\to\cal C_{0}\subset\cal C$
be a functor. Then we have a functor $\funct F\times\funct F^{op}:\cal D\times\cal D^{op}\to\cal C_{0}\times\cal C_{0}^{op}$
and we can form a bifunctor 
\[
{\rm Cohom}_{\mc C}\circ\left(\funct F\times\funct F^{op}\right):\cal D\times\cal D^{op}\to\cal C.
\]

\begin{defn}
\label{coend-defn}We define $\coendf{\funct F}{}:=\coendf{{\rm Cohom}_{\mc C}\circ\left(\funct F\times\funct F^{op}\right)}{}$.
\end{defn}
Clearly, under the assumptions above, $\coendf{\funct F}{}$ exists
if $\mc C$ is cocomplete and either $\mc C_{0}$ or $\mc D$ is small.
\begin{rem}
If $\cal C$ is a rigid category, for a functor $\funct F:\cal D\to\cal C$
we have the notion of $\coendf{\funct F}{}:={\rm coend}\left(\funct F\Cten C\funct F^{*}\right)$,
where $\funct F^{*}$ is the composition $\left(-\right)^{*}\circ\funct F$.
Since for all $X,Y\in\cal C$ $X\Cten CY^{*}\cong\coend C{X,Y}$,
this definition of coend is a special case of ours.\end{rem}
\begin{prop}
\label{prop:coend-prop}Let $\funct F:\cal D\to\cal C_{0}\subset\cal C$
be a functor and suppose that $\coendf{\funct F}{}$ exists. Then
\begin{enumerate}
\item $\coendf{\funct F}{}$ is a coalgebra in $\mc C$;
\item \label{coend_comodule_structure}$\funct F\left(X\right)$ is a comodule
over $\coendf{\funct F}{}$ for every $X\in{\rm Ob}\left(\cal D\right)$
;
\item $\funct F\left(\phi\right)$ is a $\coendf{\funct F}{}$-comodule
morphism between $\funct F\left(X\right)$ and $\funct F\left(Y\right)$
for every $\phi\in\cal D\left(X,Y\right)$ ;
\item \label{enu:the-transformation-delta}the transformation $\delta_{\funct F}:\funct F\to\funct F\Cten C\coendf{\funct F}{}$,
with $\delta_{\funct F}\left(X\right)$ being the comodule structure
from \ref{coend_comodule_structure}, is natural;
\item $\delta_{\funct F}$ is universal, i.e. any transformation $\funct F\to\funct F\Cten CM$
factors though $\delta_{\funct F}$ via morphism $\coendf{\funct F}{}\to M$.
\end{enumerate}
\end{prop}
\begin{proof}
(4) can be proven via diagram \begin{center}
\begin{tikzpicture}[node distance=2cm, auto]    
\node (FYcohXY) {$\funct F\left(Y\right)\Cten C\cohom[C]{\funct F\left(X\right)}{\funct F\left(Y\right)}$};
\node (FX) [left= and 2.5cm of FYcohXY, above of=FYcohXY] {$\funct F\left(X\right)$};
\node (FXcoendX) [right= and 2.5cm of FX] {$\funct F\left(X\right)\Cten C\coend C{\funct F\left(X\right)}$};
\node (FYcoendX) [right= and 1cm of FYcohXY] {$\funct F\left(Y\right)\Cten C\coend C{\funct F\left(X\right)}$};
\node (FY) [left= and 2.5cm of FYcohXY, below of=FYcohXY] {$\funct F\left(Y\right)$};
\node (FYcoendY) [right= and 2.5cm of FY] {$\funct F\left(Y\right)\Cten C\coend C{\funct F\left(Y\right)}$};
\node (FXcoendF) [right= and 3cm of FYcoendX, above of=FYcoendX] {$\funct F\left(X\right)\Cten C\coendf{\funct F}{}$};
\node (FYcoendF) [right= and 3cm of FYcoendX, below of=FYcoendX] {$\funct F\left(Y\right)\Cten C\coendf{\funct F}{}$};
\draw[->] (FX) to node {$\coev{\funct F\left(X\right)}$} (FXcoendX);
\draw[->] (FX) to node {$\coev{\funct F\left(X\right)}$} (FYcohXY);
\draw[->] (FXcoendX) to node {$\funct F\left(\phi\right)\Cten C\id{\coend C{\funct F\left(X\right)}}$} (FYcoendX);
\draw[->] (FXcoendX) to node {$\id{\funct F\left(X\right)}\Cten C i_X$} (FXcoendF);
\draw[->] (FYcohXY) to node {$\alpha$} (FYcoendX);
\draw[->] (FX) to node [swap] {$\funct F\left(\phi\right)$} (FY);
\draw[->] (FY) to node {$\coev{\funct F\left(Y\right)}$} (FYcoendY);
\draw[->] (FYcoendY) to node {$\id{\funct F\left(Y\right)}\Cten C i_Y$} (FYcoendF);
\draw[->] (FYcohXY) to node {$\beta$} (FYcoendY);
\draw[->] (FYcoendX) to node [swap] {$\id{\funct F\left(Y\right)}\Cten C i_X$} (FYcoendF);
\draw[->] (FXcoendF) to node {$\funct F\left(\phi\right)\Cten C\id{\coendf{\funct F}{}}$} (FYcoendF);
\draw[->, bend left] (FX) to node [swap]{$\delta_{\funct F}\left(X\right)$}(FXcoendF);
\draw[->, bend right] (FY) to node {$\delta_{\funct F}\left(Y\right)$}(FYcoendF);
\end{tikzpicture} 
\end{center}with $\alpha=\id{\funct F\left(Y\right)}\Cten C\cohom[C]{\id{\funct F\left(X\right)}}{\funct F\left(\phi\right)^{op}}$
and $\beta=\id{\funct F\left(Y\right)}\Cten C\cohom[C]{\funct F\left(\phi\right)}{\id{\funct F\left(Y\right)}}$.

Similarly one can prove that 
\[
\delta_{\funct F}^{2}:=\delta_{\funct F}\circ\delta_{\funct F}:\funct F\to\funct F\Cten C\coendf{\funct F}{}\Cten C\coendf{\funct F}{}
\]
 is also a natural transformation. Since every $\funct F\left(X\right)$
is a comodule over $\coendf{\funct F\left(X\right)}{\mc C}$, this
imply that the diagrams \begin{center}
\begin{tikzpicture}[node distance=2cm, auto]    
\node (cohXY) {$\cohom[C]{\funct F\left(X\right)}{\funct F\left(Y\right)}$};
\node (coendX) [below= and 1cm of cohXY] {$\coend C{\funct F\left(X\right)}$};
\node (coendY) [above= and 1cm of cohXY] {$\coend C{\funct F\left(Y\right)}$};   
\node (coendF) [right= and 1cm of cohXY] {$\coendf{\funct F}{}$};
\node (coendYY) [right= and 2cm of coendY] {$\coend C{\funct F\left(Y\right)}\Cten C \coend C{\funct F\left(Y\right)}$};  
\node (coendXX) [right= and 2cm of coendX] {$\coend C{\funct F\left(X\right)}\Cten C \coend C{\funct F\left(X\right)}$};  
\node (coendFF) [right= and 1.5cm of coendF] {$\coendf{\funct F}{}\Cten C \coendf{\funct F}{}$};
\draw[->] (cohXY) to node {$$} (coendY);
\draw[->] (cohXY) to node {$$} (coendX); 
\draw[->] (coendY) to node {$i_Y$} (coendF); 
\draw[->] (coendX) to node [swap] {$i_X$} (coendF); 
\draw[->] (coendX) to node [swap] {$\Delta_{\coend C{\funct F\left(X\right)}}$} (coendXX); 
\draw[->] (coendY) to node {$\Delta_{\coend C{\funct F\left(Y\right)}}$} (coendYY); 
\draw[->] (coendYY) to node {$i_Y\Cten C i_Y$} (coendFF); 
\draw[->] (coendXX) to node [swap] {$i_X\Cten C i_X$} (coendFF); 
\draw[->, dashed] (coendF) to node {$\Delta_{\coendf{\funct F}{}}$} (coendFF); 
\end{tikzpicture} 
\end{center} and \begin{center}
\begin{tikzpicture}[node distance=2cm, auto]    
\node (cohXY) {$\cohom[C]{\funct F\left(X\right)}{\funct F\left(Y\right)}$};
\node (coendX) [below= and 1.5cm of cohXY, right of=cohXY] {$\coend C{\funct F\left(X\right)}$};
\node (coendY) [above= and 1.5cm of cohXY, right of=cohXY] {$\coend C{\funct F\left(Y\right)}$};   
\node (coendF) [right= and 1cm of cohXY] {$\coendf{\funct F}{}$};     
\node (coendFF) [right= and 1.5cm of coendF] {$I$};
\draw[->] (cohXY) to node {$$} (coendY);
\draw[->] (cohXY) to node {$$} (coendX); 
\draw[->] (coendY) to node [swap] {$i_Y$} (coendF); 
\draw[->] (coendX) to node  {$i_X$} (coendF); 
\draw[->, bend right] (coendX.0) to node [swap] {$\epsilon_{\coend C{\funct F\left(X\right)}}$} (coendFF); 
\draw[->, bend left] (coendY.0) to node {$\epsilon_{\coend C{\funct F\left(Y\right)}}$} (coendFF); 
\draw[->, dashed] (coendF) to node {$\epsilon_{\coendf{\funct F}{}}$} (coendFF); 
\end{tikzpicture} 
\end{center}are the diagrams of dinatural transformations (note that $\coendf{\funct F}{}\Cten C\coendf{\funct F}{}$
does not have to be the colimit of $\coend C{\funct F\left(X\right)}\Cten C\coend C{\funct F\left(X\right)}$).
Thus $\cm_{\coendf{\funct F}{}}$ and $\epsilon_{\coendf{\funct F}{}}$
and coalgebra structure is induced by the one of $\coendf{\funct F\left(X\right)}{\mc C}$,
which proves (1). The check of the coalgebra axioms is straightforward.

For (2), the $\coendf{\funct F}{}$-comodule structure on $\funct F\left(X\right)$
is given by 
\[
\delta_{\funct F\left(X\right)}:=\left(\id{\funct F\left(X\right)}\Cten Ci_{X}\right)\circ\coev{\funct F\left(X\right)}.
\]
Again, the check of the axioms is straightforward. (3) is equivalent
to (4). (5) follows from the universal property of $\coendf{\funct F}{}$.
\end{proof}

\subsection{Relation to functor ${\rm Nat}_{\mc{}}\left(\funct F,\funct F\ten-\right)$. }

Now let $\mc C$ be a locally small category. Then natural transformations
${\rm Nat}_{\mc{}}\left(\funct F,\funct F\Cten CM\right)$ form a
functor ${\rm Nat}_{\mc{}}\left(\funct F,\funct F\Cten C-\right):\mc C\to{\rm Set}$.
In \cite{PAR2} for a functor $\funct F:\cal D\to\cal C$ of one variable
$\coendf{\funct F}{}$ is defined as a representing object of ${\rm Nat}_{\mc{}}\left(\funct F,\funct F\Cten C-\right)$.
This definition is more general than \ref{coend-defn}, since one
might not be able to form the bifunctor ${\rm Cohom}\circ\left(\funct F\times\funct F^{op}\right)$.
The following results show that when one can - the two definitions
agree. Thus one might think of the definition \ref{coend-defn} as
of the way to compute $\coendf{\funct F}{}$ when the target category
is coclosed.
\begin{prop}
\label{prop:nat-cow}There is a one-to-one correspondence between
elements of ${\rm Nat}_{\mc{}}\left(\funct F,\funct F\Cten CM\right)$
and cowedges from ${\rm Cohom}\circ\left(\funct F\times\funct F^{op}\right)$
to $M$.\end{prop}
\begin{proof}
Since we have adjunction $\cohom[C]--\dashv\left(-\ten-\right)$,
from correspondences 
\[
\cal C\left(\cohom[C]{\funct F\left(X\right)}{\funct F\left(X\right)},M\right)\leftrightarrow\cal C\left(\funct F\left(X\right),\funct F\left(X\right)\Cten CM\right)
\]
 one has a one-to-one correspondence between the families of morphisms
\[
\left\{ \mu_{X}:\funct F\left(X\right)\to\funct F\left(X\right)\Cten CM\right\} 
\]
 and the families of morphisms 
\[
\left\{ \mu_{X}':\cohom[C]{\funct F\left(X\right)}{\funct F\left(X\right)}\to M\right\} .
\]
One can check directly that naturality condition on $\left\{ \mu_{X}\right\} $
implies dinaturality condition on $\left\{ \mu_{X}'\right\} $ and
vice verse via the diagram\begin{center}
\begin{tikzpicture}[node distance=1.5cm, auto]    
\node (fycohfxfy) {$\funct F\left(Y\right)\Cten C\cohom[C]{\funct F\left(X\right)}{\funct F\left(Y\right)}$};
\node (fx) [right=  and 0.5cm of fycohfxfy] {$\funct F\left(X\right)$};
\node (fxcohfx) [left= and 1cm of fx, above of=fx] {$\funct F\left(X\right)\Cten C\coend C{\funct F\left(X\right)}$};
\node (fxm) [right= and 2cm of fx] {$\funct F\left(X\right)\Cten CM$};
\node (fycohfx) [right= and 2cm of fxcohfx] {$\funct F\left(Y\right)\Cten C\coend C{\funct F\left(X\right)}$};
\node (fy) [below of=fx] {$\funct F\left(Y\right)$};
\node (fym) [below of=fxm] {$\funct F\left(Y\right)\Cten CM$};
\node (fycohfy) [below= and 0.5cm of fy] {$\funct F\left(Y\right)\Cten C\coend C{\funct F\left(Y\right)}$};
\draw[->] (fx) to node {$$} (fycohfxfy);
\draw[->] (fx) to node {$\mu_{X}$} (fxm);
\draw[->] (fx) to node {$\funct F\left(f\right)$} (fy);
\draw[->] (fy) to node {$\mu_{Y}$} (fym);
\draw[->] (fxm) to node [swap] {$\funct F\left(f\right)\Cten C\id M$} (fym);
\draw[->] (fx) to node {$$} (fxcohfx);
\draw[->] (fy) to node {$$} (fycohfy);
\draw[->] (fxcohfx) to node [right=15pt] {$\id{\funct F\left(X\right)}\Cten{C}\mu_X'$} (fxm);
\draw[->,bend right] (fycohfy.1) to node [swap] {$\id{\funct F\left(Y\right)}\Cten{C}\mu_Y'$} (fym);
\draw[->] (fxcohfx) to node[above=3pt] {$\funct F\left(f\right)\Cten{C}\id{\coend C{\funct F\left(X\right)}}$} (fycohfx);
\draw[->, bend left] (fycohfx) to node {$\id{\funct F\left(Y\right)}\Cten{C}\mu_X'$} (fym);
\draw[->] (fycohfxfy.188) .. controls +(270:3cm) and +(180:6cm) .. node {$\id{\funct F\left(Y\right)}\Cten{C}\cohom[C]{\funct F\left(f\right)}{\id{\funct F\left(Y\right)}}$} (fycohfy);
\draw[->] (fycohfxfy.120) .. controls +(90:3cm) and +(140:1.5cm) .. node {$\id{\funct F\left(Y\right)}\Cten{C}\cohom[C]{\id{\funct F\left(X\right)}}{\funct F\left(f\right)^{op}}$} (fycohfx);
\end{tikzpicture} 
\end{center}\end{proof}
\begin{cor}
Suppose $\coendf{\funct F}{}$ exists. Then \textup{${\rm Nat}_{\mc{}}\left(\funct F,\funct F\Cten C-\right)$}
is representable by $\coendf{\funct F}{}$;\end{cor}
\begin{proof}
Follows from the universal property of $\coendf{\funct F}{}$.
\end{proof}

\subsection{$\cal C$-cowedges of a $\cal C$-functor and $\coendf{\funct F}{\cal C}$.}

Let $\cal A$ and $\cal B$ be $\cal C$-categories, $\cal B_{0}\subset\cal B$
coclosed in $\cal B$ and $\funct F:\cal A\to\cal B_{0}\subset\cal B$
be a $\cal C$-functor. Similar to the proof of the proposition \ref{prop:nat-cow},
for a $\mc C$-natural transformation $\mu:\funct F\to\funct F\Cten BM$
we can expand the diagram, expressing it's $\mc C$-naturality, to
the diagram \begin{center}
\begin{tikzpicture}[node distance=1.5cm, auto]    
\node (fycohfxfy) {$C\Cten{CB}\funct F\left(X\right)\Cten{B}\cohom[B]{\funct F\left(C\Cten{CA}X\right)}{C\Cten{CB}\funct F\left(X\right)}$};
\node (fxcohfx) [below of=fycohfxfy] {$\funct F\left(C\Cten{CA}X\right)\Cten B\coend B{\funct F\left(C\Cten{CA}X\right)}$};
\node (fx) [below of=fxcohfx, left= and 0cm of fxcohfx] {$\funct F\left(C\Cten{CA}X\right)$};
\node (fxm) [right= and 2cm of fx] {$\funct F\left(C\Cten{CA}X\right)\Cten BM$};
\node (fy) [below of=fx] {$C\Cten{CB}\funct F\left(X\right)$};
\node (fycohfy) [right= and 1.5cm of fy, below of=fy] {$C\Cten{CB}\funct F\left(X\right)\Cten B\coend B{\funct F\left(X\right)}$};
\node (fycohfx) [right= and 0cm of fxm] {$C\Cten{CB}\funct F\left(X\right)\Cten B\coend B{\funct F\left(C\Cten{CA}X\right)}$};
\node (fym) [below of=fxm] {$C\Cten{CB}\funct F\left(X\right)\Cten BM$};
\draw[->, bend left] (fx) to node {$$} (fycohfxfy);
\draw[->] (fx) to node {$\scriptstyle \mu_{C\ten_{\mc CA}X}$} (fxm);
\draw[->] (fx) to node [swap]{$\scriptstyle \xi$} (fy);
\draw[->] (fy) to node {$\scriptstyle \id{C}\ten_{\mc CB} \mu_{X}$} (fym);
\draw[->] (fxm) to node [swap]{$\scriptstyle \xi\ten_{\mc B}\id M$} (fym);
\draw[->] (fx) to node {$$} (fxcohfx.185);
\draw[->] (fy) to node {$$} (fycohfy.165);
\draw[->] (fxcohfx) to node [right=1pt] {$\scriptstyle \id{\funct F\left(C\ten_{\mc CA}X\right)}\ten_{\mc B}\mu'_{C\ten_{\mc CA}X}$} (fxm);
\draw[->,bend right] (fycohfy.20) to node {$\scriptstyle \id{C\ten_{\mc CB}\funct F\left(X\right)}\ten_{\mc B}\mu_X'$} (fym);
\draw[->] (fxcohfx.-5) to node[above=3pt] {$\scriptstyle \xi\ten_{\mc B}\id{\coend B{\funct F\left(C\ten_{\mc CA} X\right)}}$} (fycohfx.10);
\draw[->, bend left=10] (fycohfx) to node {$\scriptstyle\id{C\ten_{\mc CB}\funct F\left(X\right)}\ten_{\mc B}\mu'_{C\ten_{\mc CA}X}$} (fym.0);
\draw[->] (fycohfxfy.182) .. controls +(220:2cm) and +(140:1cm) .. node {$$} (fycohfy.178);
\draw[->] (fycohfxfy) .. controls +(-5:4cm) and +(120:1.5cm) .. node {$$} (fycohfx.7);
\end{tikzpicture} 
\end{center}which motivates the following definition. Denote the induced map 
\[
\lambda_{C,X}:\cohom[B]{\funct F\left(C\Cten{CA}X\right)}{C\Cten{CB}\funct F\left(X\right)}\to\coend B{\funct F\left(X\right)}.
\]

\begin{defn}
We say that a cowedge $\mu':\cohom{\funct F}{\funct F^{op}}\overset{\cdot\cdot}{\to}d$
is a $\cal C$-cowedge over $\funct F$ if for all $X\in\mc A$, $C\in\mc C$
the following diagram \begin{center}
\begin{tikzpicture}[node distance=2cm, auto]    
\node (cohXY) {$\cohom[B]{\funct F\left(C\Cten{CA} X\right)}{{C\Cten{CB}\funct F}^{op}\left(X\right)}$};
\node (coendX) [below of=cohXY, right of=cohXY] {$\coend B{\funct F\left(X\right)}$};
\node (coendY) [above of=cohXY, right of=cohXY] {$\coend B{\funct F\left(C\Cten{CA} X\right)}$};   
\node (coendFF) [right= and 1cm of cohXY] {$d$};
\draw[->] (cohXY) to node {$\cohom[B]{\id{\funct F\left(C\Cten{CA}X\right)}}{\xi^{op}}$} (coendY);
\draw[->] (cohXY) to node [swap]{$\lambda_{C,X}$} (coendX); 
\draw[->] (coendX) to node [swap] {$\mu'_{\coend C{\funct F\left(X\right)}}$} (coendFF); 
\draw[->] (coendY) to node {$\mu'_{\coend C{\funct F\left(C\Cten{CA}X\right)}}$} (coendFF); 
\end{tikzpicture} 
\end{center}is commutative.
\end{defn}
Similar to the section \ref{sub:Coendomorphism-of-bifunctors}, $\mc C$-cowedges
form a subcategory of the category of cowedges over $\funct F$. 
\begin{defn}
Similar to the definition \ref{end-coend-defn} we define $\coend C{\funct F}$
as the initial object in this category.
\end{defn}
The diagram in the beginning of this section proves the following
result.
\begin{lem}
Let $\funct F:\cal A\to\cal B_{0}\subset\cal B$ be a $\cal C$-functor.
Then under the correspondence from the proposition \ref{prop:nat-cow}
$\mc C$-cowedges correspond precisely to $\mc C$-natural transformations.\end{lem}
\begin{cor}
If $\mc B$ is cocomplete and either $\mc A$ or $\mc B$ is small
then $\coendf{\funct F}{\mc C}$ exists for every $\cal C$-functor
$\funct F$ (similar to section \ref{sub:subdivision-category,-existence}).
\end{cor}
Similar to the section \ref{sub:subdivision-category,-existence},
one can construct $\coend C{\funct F}$ explicitly as a coequalizer.
Namely, 
\[
\underset{q}{F_{1}\coprod F_{2}\overset{p}{\rightrightarrows}}\coprod_{X\in{\rm Ob}\left(\mc A\right)}\coend B{\funct F\left(X\right)}\to\coendf{\funct F}{\mc B},
\]
where 
\[
F_{1}:=\coprod_{f\in{\rm Mor}\left(\mc A\right)}\cohom[B]{\funct F\left({\rm dom}\left(f\right)\right)}{\funct F\left({\rm codom}\left(f\right)\right)}
\]
 and 
\[
F_{2}:=\coprod_{X\in{\rm Ob}\left(\mc A\right),C\in{\rm Ob}\left(\mc C\right)}\cohom[B]{\funct F\left(C\Cten{CA}X\right)}{C\Cten{CB}\funct F\left(X\right)}.
\]
The equalized pair is defined by equations 
\[
p\circ i_{f}=i_{{\rm dom}\left(f\right)}\circ\cohom[B]{\id{\funct F\left({\rm dom}\left(f\right)\right)}}{\funct F\left(f\right)^{op}},
\]
 
\[
q\circ i_{f}=i_{{\rm codom}\left(f\right)}\circ\cohom[B]{\funct F\left(f\right)}{\id{\funct F\left({\rm codom}\left(f\right)\right)}}
\]
 and by 
\[
p\circ i_{C,X}=i_{C\Cten{CA}X}\circ\left(\cohom[B]{\id{\funct F\left(C\Cten{CA}X\right)}}{\xi^{op}}\right),
\]
\[
q\circ i_{C,X}=i_{X}\circ\lambda_{C,X}.
\]

\smallskip{}

\begin{cor}
Suppose $\coendf{\funct F}{\mc C}$ exists. Then
\begin{enumerate}
\item \textup{${\rm Nat}_{\mc C}\left(\funct F,\funct F\Cten B-\right)$
}is representable by $\coendf{\funct F}{\mc C}$;
\item $\delta_{\funct F,\mc C}:\funct F\to\funct F\Cten B\coendf{\funct F}{\mc C}$
is universal, i.e. any $\mc C$-transformation $\funct F\to\funct F\Cten BM$
factors though $\delta_{\funct F,\mc C}$ via morphism $\coendf{\funct F}{\mc C}\to M$.
\end{enumerate}
\end{cor}
\begin{rem}
If $C\in\cal C$ is a coalgebra and $\funct F:{\rm Comod}_{\cal C}-C\to\cal C$
is a forgetful functor from the category of right comodules over $C$,
then $\nat{\mc C}{\funct F}{\funct F\Cten B-}$ is multirepresentable
(similar to \cite[3.8.6]{PAR1}).
\end{rem}

\subsection{Reconstruction theorems.}

We slightly reformulate the ``restricted'' reconstruction theorems.
\begin{prop}
\label{prop:reconstruction}Let $\cal C$ be a braided monoidal category
and $\cal C_{0}$ be a full braided monoidal subcategory.
\begin{itemize}
\item Let $C$ be a $\cal C_{0}-I$-generated coalgebra in $\cal C$ and
$\funct F:\comod{\cal C_{0}}-C\to\cal C_{0}\subset\cal C$ be the
forgetful functor. Then we have an isomorphism of coalgebras $\coendf{\funct F}{\mc C_{0}}\cong C$
\cite[4.3]{PAR2};
\item Let $\cal C$ also have colimits of $I$-diagrams and let $\ten$
preserve those colimits in both variables (and thus preserve $I-I$-colimits).
Let $C$ also be a $\cal C_{0}$-central bialgebra \emph{(}and thus
$\comod{\cal C_{0}}-C$ is a monoidal category\emph{)}. Then $\nat{\cal C_{0}}{\funct F}{\funct F\Cten C-}$
is multirepresentable by $C$ and $\coendf{\funct F}{\mc C_{0}}\cong C$
is the isomorphism of bialgebras \cite[4.5]{PAR2}.
\end{itemize}
\end{prop}
\begin{rem}
It is often possible to reconstruct the bialgebra structure without
additional assumptions on $\mc C$ and $\ten$ (see \ref{sub:Banach-reconstruction-for}).
\end{rem}

\subsection{Recognition theorem}

From our previous results we know that for a $\mc C_{0}$-functor
$\funct F:\cal B\to\cal C_{0}\subset\cal C$ if the category $\mc C$
is cocomplete, $\mc C_{0}\subset\mc C$ is coclosed in $\mc C$ and
either $\mc B$ or $\mc C_{0}$ is small, then $\coend{}{\funct F}$
exists. Together with the proposition \ref{prop:equivalence-of-cats}
this solves the recognition problem for coalgebras. For the recognition
of bialgebras we can simply use the corresponding part of the proposition
\ref{prop:The-functor-Nat-properties}.

We combine these statements in the following proposition (compare
with \cite[4.7]{PAR2}).
\begin{prop}
\label{prop:recognition}Let ${\cal A}$, ${\cal B}$ be ${\cal C}$-categories,
$\mc A$ be locally small, cocomplete and $\mc C$-monoidal, $\mc A_{0}\subset\mc A$
be coclosed in $\mc A$, $\funct F:{\cal B\to{\cal A}}_{0}$ be a
${\cal C}$-functor. Let also $\mc A_{0}$ or $\mc B$ be small. Then
\begin{enumerate}
\item $A:=\coendf{\funct F}{\mc C}$ exists and $\funct F$ factors as $\funct F=\funct U\circ\funct I_{\funct F}$,
with $\funct I_{\funct F}:\mc B\to\comodc{A_{{\rm 0}}}-A$, $\funct U:\comodc{A_{{\rm 0}}}-A\to\mc A_{0}$
being the forgetful functor;
\item let ${\cal A}$ be ${\cal C}$-braided ${\cal C}$-monoidal, $\mc B$
is ${\cal C}$-monoidal, $\funct F$ is ${\cal C}$-monoidal functor
and $\nat{\cal C}{\funct F}{\funct F\ten-}$ be multirepresentable.
Then $\coend C{\funct F}$ is a bialgebra in $\mc A$, $\comodc{A_{{\rm 0}}}-A$
has natural $\mc C$-monoidal structure and $\funct I_{\funct F}$ is
a ${\cal C}$-monoidal functor;
\end{enumerate}
\end{prop}

\subsection{\textmd{\label{sub:Equivalence-of-categories}}Equivalence of categories}

Let $\mc C$ be a monoidal category, $\mc C_{0}\in\mc C$ be a full
monoidal subcategory and $C\in\mc C$ be a coalgebra. 

Let $\mc A$ be a $\mc C_{0}$-category and $\funct F:\mc A\to\mc C_{0}$
be a $\mc C_{0}$-functor. If $C=\coendf{\funct F}{\mc C_{0}}$ exists
then we have a $\mc C_{0}$-functor 
\[
\funct I_{\funct F}:\mc A\to\comod{\mc C_{0}}-C
\]
and $\funct F$ factors as $\funct F=\funct U\circ\funct I_{\funct F}$
with $\funct U:\comodc{C_{{\rm 0}}}-C\to\mc C_{0}$ being the forgetful
functor. If $\funct F$ is faithful then $\funct I_{\funct F}$ is an
embedding of categories. Let's check when it is a category equivalence.

First note that, since $\funct F$ is a $\mc C_{0}$-functor,
we always have an isomorphism 
\[
\funct F\left(A\right)\Cten C\funct F\left(A\right)\cong\funct F\left(\funct F\left(A\right)\Cten{C_{{\rm 0}}A}A\right).
\]

\begin{thm}
\label{prop:equivalence-of-cats}Suppose $C$ is $\mc C_{0}$-$I$-generated
by a filtered system
of coalgebras $\funct F(A_i)=\left(\funct F(A_i),\cm_{\funct F(A_i)},\epsilon_{\funct F(A_i)}\right)$, 
such that $\cm_{\funct F\left(A_{i}\right)}\in\funct F\left(\mc A\left(A_{i},\funct F\left(A_{i}\right)\Cten{C_{{\rm 0}}A}A_{i}\right)\right)$.
Let also $\mc A$ have equalizers. Then

If $\funct F$ preserves equalizers then $\funct I_{\funct F}$ is essentially surjective.

If $\funct F$ is faithful and reflects equalizers then $\funct I_{\funct F}$ is fully faithful.

\end{thm}
\begin{proof}
Let $\left(M,\rho_{M}\right)\in\comod{\mc C_{0}}-C$. Since $C$ is
$\mc C_{0}-I$-generated, $\left(M,\rho_{M}\right)\in\comod{\mc C_{0}}-\funct F\left(A_{i}\right)$
for some $i\in I$. Since $M$ is a comodule, it is an equalizer of
the pair \begin{center}
\begin{tikzpicture}[node distance=1.5cm, auto]    
\node (mai) {$M\Cten C\funct F\left(A_{i}\right)$};
\node (maiai) [right= and 3cm of mai] {$M\Cten C\funct F\left(A_{i}\right)\Cten C\funct F\left(A_{i}\right).$};
\draw[->] (mai.5) to node {$\rho_{M}\Cten C\id{\funct F\left(A_{i}\right)}$} (maiai.177);
\draw[->] (mai.355) to node [swap] {$\id M\Cten C\cm_{\funct F(A_{i})}$} (maiai.183);
\end{tikzpicture} 
\end{center}Since $\funct F$ is $\mc C_{0}$-functor, the top morphism is actually 
\begin{center}
\begin{tikzpicture}[node distance=1.5cm, auto]    
\node (mai) {$\funct F\left(M\Cten{C_{{\rm 0}}A}A_{i}\right)$};
\node (maiai) [right= and 3cm of mai] {$\funct F\left(\left(M\Cten{C_{{\rm 0}}}\funct F\left(A_{i}\right)\right)\Cten{C_{{\rm 0}}A}A_{i}\right)$};
\draw[->] (mai) to node {$\funct F\left(\rho_{M}\Cten{C_{\rm 0}A}\id{A_{i}}\right)$} (maiai);
\end{tikzpicture} 
\end{center}and the bottom one is \begin{center}
\begin{tikzpicture}[node distance=1.5cm, auto]    
\node (mai) {$\funct F\left(M\Cten{C_{{\rm 0}}A}A_{i}\right)$};
\node (maiai) [right= and 3cm of mai] {$\funct F\left(M\Cten{C_{{\rm 0}}A}\left(\funct F\left(A_{i}\right)\Cten{C_{{\rm 0}}A}A_{i}\right)\right).$};
\draw[->] (mai) to node {$\funct F\left(\id M\Cten{C_{{\rm 0}}A}\tilde{\cm}_{A_{i}}\right)$} (maiai);
\end{tikzpicture} 
\end{center}Since $\mc A$ has equalizers, there exists an object $\tilde{M}\in\mc A$
satisfying the diagram \begin{center}
\begin{tikzpicture}[node distance=1.5cm, auto]    
\node (m) {$\tilde M$};
\node (mai) [right= and 1cm of m] {$M\Cten{C_{{\rm 0}}A}A_{i}$};
\node (maiai) [right= and 3cm of mai] {$M\Cten{C_{{\rm 0}}A}\left(\funct F\left(A_{i}\right)\Cten{C_{{\rm 0}}A}A_{i}\right).$};
\draw[->] (mai.355) to node [swap] {$\id M\Cten{C_{{\rm 0}}A}\tilde{\cm}_{A_{i}}$} (maiai.182);
\draw[->] (mai.5) to node {$\rho_{M}\Cten{C_{\rm 0}A}\id{A_{i}}$} (maiai.178);
\draw[->] (m) to node {$\tilde \rho_M$} (mai);
\end{tikzpicture} 
\end{center}Since $\funct F$ and the forgetful functor $\funct U:\comod{\mc C_{0}}-C\to\mc C_{0}$
both preserve equalizers ($\funct U$ creates finite limits), then
so is $\funct I_{\funct F}$ and thus $\funct I_{\funct F}\left(\tilde{M}\right)= (M,\rho_M)$.
Thus $\funct I_{\funct F}$ is essentially surjective.

$\funct I_{\funct F}$ is faithful since $\funct F$ is.

Let $M,N\in\comod{\mc C_{0}}-C$
and let $f:M\to N$ be a $C$-comodule morphism. Since $C$ is $\mc C_{0}-I$-generated
and the system $\left\{\funct F\left(A_{i}\right)\right\} _{i\in I}$
is filtered, $M,N\in\comod{\mc C_{0}}-\funct F\left(A_{i}\right)$
for some $i\in I$ and we have a diagram \begin{center}
\begin{tikzpicture}[node distance=1.5cm, auto]    
\node (m) {$M$};
\node (mai) [right= and 1cm of m] {$M\Cten C\funct F\left(A_{i}\right)$};
\node (n) [below of=m]{$N$};
\node (nai) [below of =mai] {$N\Cten C\funct F\left(A_{i}\right)$};
\draw[->] (m) to node {$\rho_M$} (mai);
\draw[->] (n) to node {$\rho_N$} (nai);
\draw[->] (m) to node {$f$} (n);
\draw[->] (mai) to node {$f\Cten C\id{\funct F\left(A_{i}\right)}$} (nai);
\node (naiai) [right= and 3cm of nai] {$N\Cten C\funct F\left(A_{i}\right)\Cten C\funct F\left(A_{i}\right)$};
\draw[->] (nai.5) to node {$\rho_{N}\Cten C\id{\funct F\left(A_{i}\right)}$} (naiai.177);
\draw[->] (nai.355) to node [swap] {$\id N\Cten C\cm_{F(A_{i})}$} (naiai.183);
\end{tikzpicture} 
\end{center}where $\left(f\Cten C\id{\funct F\left(A_{i}\right)}\right)\circ\rho_{M}$
equalizes the kernel pair. Since $\funct I_{\funct F}$ is faithful,
in the diagram \begin{center}
\begin{tikzpicture}[node distance=1.5cm, auto]    
\node (m) {$\tilde M$};
\node (mai) [right= and 1cm of m] {$M\Cten{C_{{\rm 0}}A}A_{i}$};
\node (n) [below of=m] {$\tilde N$};
\node (nai) [below of=mai] {$N\Cten{C_{{\rm 0}}A}A_{i}$};
\node (naiai) [right= and 3cm of nai] {$N\Cten{C_{{\rm 0}}A}\left(\funct F\left(A_{i}\right)\Cten{C_{{\rm 0}}A}A_{i}\right).$};
\draw[->] (nai.355) to node [swap] {$\id N\Cten{C_{{\rm 0}}A}\tilde{\cm}_{A_{i}}$} (naiai.182);
\draw[->] (nai.5) to node {$\rho_{N}\Cten{C_{\rm 0}A}\id{A_{i}}$} (naiai.178);
\draw[->] (n) to node {$\tilde \rho_N$} (nai);
\draw[->] (m) to node {$\tilde \rho_M$} (mai);
\draw[->] (mai) to node {$f\Cten{C_{{\rm 0}}A}\id{A_{i}}$} (nai);
\draw[->, dashed] (m) to node {$\tilde f$} (n);
\end{tikzpicture} 
\end{center}$\left(f\Cten{C_{{\rm 0}}A}\id{A_{i}}\right)\circ\tilde{\rho}_{M}$
also equalizes the kernel pair. This implies the existence of the
map $\tilde{f}:\tilde{M}\to\tilde{N}$. Since $\funct F$ reflects equalizers, if $\funct F(M')=M$ then $M'\cong\tilde{M}$. Thus $\funct I_{\funct F}$
is full and is a category equivalence.\end{proof}

\section{Applications to p-adic representations: reconstruction for coalgebras
of compact type}

We will modify previous construction and apply it to the reconstruction
of topological coalgebras, which are, as a locally convex $K-$vector
space, of compact type (LS-spaces).

Recall \cite[1.1.35]{EM} that in the category of locally convex topological
vector spaces $\underline{LCTVS}$ for any $U$, $V$ and $W$ the
hom-tensor adjunction of vector spaces gives rise to the continuous
bijection 
\[
L_{b}\left(U\ten_{K,i}V,W\right)\simeq L_{b}\left(U,L_{b}\left(V,W\right)\right)
\]
under the condition that $U$ is barreled. Since, by \cite[18.8]{nfa},
for an $LS$-space $V$ and a complete space $W$ we have a topological
isomorphism $V\cten_{K,\pi}W\cong L_{b}\left(V_{b}',W\right)$, if
$U$ is also an LS-space, we the following identity 
\[
L_{b}\left(U\cten_{K,i}V_{b}',W\right)\simeq L_{b}\left(U,V\cten_{K,\pi}W\right).
\]

Let $\underline{CHLCTVS}$ be the category of complete Hausdorff locally
convex topological vector spaces. It is a symmetric monoidal category,
with monoidal structure given by complete projective tensor product
$\cten_{K,\pi}$. LS-spaces form a monoidal subcategory $\underline{LS}$
of $\underline{CHLCTVS}$. Since, in general, inductive and projective
tensor product topologies do not coincide, the above identity shows
that in $\underline{CHLCTVS}$ two LS-spaces $U$ and $V$ have a
cohomomorphism object $\cohomm[\underline{CHLCTVS}]UV\cong U\cten_{K,i}V_{b}'$,
although $U$ and $V$ might not be rigid objects. Also note that
the category $\underline{CHLCTVS}$ is cocomplete, with the colimit
$\widehat{V}=\underline{CHLCTVS}-\ilim V_{i}$ of a system $\left\{ V_{i}\right\} _{i\in I}$
being the Hausdorff completion of the locally convex inductive limit
$V=\underline{LCTVS}-\ilim V_{i}$.

Before considering reconstruction theorems, lets recall that in the
algebraic case, i.e. when $\mc C$ is a full monoidal subcategory
of the category $R-{\rm mod}_{fg}$ of finitely generated modules
over a commutative ring $R$, every natural transformation $\phi$
of $\mc C$-functors $\funct F,\funct F':\mc A\to\mc C$ is $\mc C$-natural
(\cite[6.4]{PAR2}). This imply that in this case the coendomorphism
object $\coendf{\funct F}{\mc C}$ does not depend on the control
category $\mc C$ and we may simply consider $\coend{}{\funct F}$.
The categories we consider share the same property. 
\begin{lem}
Let $\cal C$ be a full monoidal subcategory of $\Ban$ or $\underline{{\rm LCTVS}}$
and $\cal A$ be a $\cal C$-category. Let $\funct F,\funct F':\cal A\to\cal C$
be $\cal C$-functors. Then every natural transformation $\phi:\funct F\overset{\cdot}{\to}\funct F'$
is $\cal C$-natural.
\end{lem}
\begin{proof}
Similar to \cite[6.4]{PAR2} one can prove that the identity, required
for $\phi$ to be $\cal C$-natural, holds on elementary tensors $x\Cten Ca\in C\Cten C\funct F\left(A\right)$,
$C\in\mc C$, $A\in\mc A$. Since all maps involved there are continuous,
the identity holds in general.
\end{proof}

\subsection{Reconstruction}

Let $\mc C$ denote the category $\underline{CHLCTVS}$ and $\underline{LS}$
denote the category of LS-spaces.

\subsubsection{Full reconstruction and recognition}

For any two monoidal subcategories $\mc C\subset\mc D\subset\underline{CHLCTVS}$
the category $\mc D$ is naturally a $\mc C$-category. If $A\in\mc D$
is a coalgebra in $\mc D$, then the categories $\comodc D-A$ and
$\comodc C-A$ are naturally $\mc C$-categories. Proposition \ref{prop:reconstruction}
can be applied to produce the following result.
\begin{prop}
(Reconstruction theorem)
\begin{itemize}
\item Let $C$ be a coalgebra in $\underline{LS}$ and $\funct F:\comod{\underline{LS}}-C\to\underline{LS}$
be the forgetful functor. Then we have an isomorphism of coalgebras
$\coendf{\funct F}{}\cong C$;
\item Let $C$ be a bialgebra \emph{(}and thus $\comod{\underline{LS}}-C$
is a monoidal category\emph{)}. Then $\nat{}{\funct F}{\funct F\Cten C-}$
is multirepresentable by $C$ and $\coendf{\funct F}{}\cong C$ is
the isomorphism of bialgebras.
\end{itemize}
\end{prop}
The recognition theorem in full setting will have the following form
(one of few possible).
\begin{prop}
Let ${\cal B}$ be an $\underline{LS}$-category and $\funct F:{\cal B}\to\underline{LS}$
be an $\underline{LS}$-functor. Then
\begin{enumerate}
\item $A:=\coendf{\funct F}{}$ exists in $\underline{CHLCTVS}$ and $\funct F$
factors as $\funct F=\funct U\circ\funct I_{\funct F}$, with $\funct I_{\funct F}:\mc B\to\comodc{\underline{\mathit{LS}}}-A$,
$\funct U:\comodc{\mathit{\underline{LS}}}-A\to\underline{LS}$ being
the forgetful functor;
\item let $\mc B$ is $\underline{LS}$-monoidal, $\funct F$ is $\underline{LS}$-monoidal
functor and $\nat{\underline{LS}}{\funct F}{\funct F\ten-}$ be multirepresentable.
Then $A$ is a bialgebra in $\underline{CHLCTVS}$, $\comodc{\underline{\mathit{LS}}}-A$
has natural $\underline{LS}$-monoidal structure and $\funct I_{\funct F}$
is a $\underline{LS}$-monoidal functor;
\item if $A$ is an LS-space and $A\in\funct F\left(\mc B\right)$ with
$\cm_{A}\in\funct F\left({\rm Mor}\left(\mc B\right)\right)$, $\mc B$
has equalizers and $\funct F$ is faithful and preserve and reflects them, then
$\funct I_{\funct F}$ is a category equivalence.
\end{enumerate}
\end{prop}
\begin{proof}
Direct application of the recognition theorem \ref{prop:recognition}
and equivalence theorem \ref{prop:equivalence-of-cats}. The only
thing one needs to prove is that the category $\underline{LS}$ is
small. This follows from the facts that LS-spaces are nuclear (\cite[11.3.5.ix]{PGSch})
and that all nuclear spaces are isomorphic to a subspace of some power
of the ``universal nuclear space'' (\cite[8.8.3]{PGSch}).\end{proof}
\begin{rem}
One can also argue directly that, since any LS-space $V$ is of countable
type, if we consider $X=K^{\mathbb{N}}$ then we have an embedding
of sets $V\subset X$ and for the topology $\tau_{V}$ of $V$ we
have $\tau_{V}\subset2^{X}$.
\end{rem}

\subsubsection{\label{sub:Banach-reconstruction-for}Banach reconstruction for coalgebras
of compact type.}

Recall that we call $C\in\underline{LS}$ (in terminology of \cite{L})
a coalgebra of compact type (or CT-$\cten$-coalgebra) if it is a
compact inductive limit of an inductive sequence of Banach $\cten$-coalgebras.
Denote by $\underline{\Ban}\subset\underline{CHLCTVS}$ the image
of the $\Ban$ under the forgetful functor into the category $\underline{CHLCTVS}$
(forgets the norm, but remembers topology). Another reconstruction
theorem we would like to consider is the reconstruction of $C$ from
the category $\comod{\underline{\Ban}}-C$.

If $C\cong\ilim C_{i}$ and $\left(V,\rho_{V}\right)\in\comod{\underline{\Ban}}-C$
then the coaction $\rho_{V}:V\to V\cten C$ send $V$ into the space
$V\cten C$. Since $V\cten C\cong\ilim\left(V\cten C_{i}\right)$
\cite[3.4.6]{L} is a regular LB-space (follows from splitting lemma
\cite[5.11]{L}), $\rho_{V}$ must factor through $V\cten C_{i}$
for some $i$, which means $\left(V,\rho_{V}\right)\in\comod{\underline{\Ban}}-C_{i}$.
This imply that $C$ is $\underline{\Ban}-\mathbb{N}$-generated in
$\underline{CHLCTVS}$ and thus we get the reconstruction for coalgebra
structure. 
\begin{prop}
Let $C$ be a CT-$\cten$-coalgebra and $\funct F:\comod{\underline{\Ban}}-C\to\underline{\Ban}\subset\underline{CHLCTVS}$
be the forgetful functor. Then we have an isomorphism of coalgebras
$\coendf{\funct F}{}\cong C$.
\end{prop}
For the reconstruction for bialgebra structure, since the tensor product
in $\underline{CHLCTVS}$ does not satisfy the conditions of the proposition
\ref{prop:reconstruction}.2, we cannot apply it here. Instead, we
first need to prove the following lemma.
\begin{lem}
\label{banCTmultirepr}Let $C$ be a CT-$\cten$-coalgebra and $\funct F:\comod{\underline{\Ban}}-C\to\underline{\Ban}\subset\underline{CHLCTVS}$
be the forgetful functor. Then the functor $\nat{\underline{\Ban}}{\funct F}{\funct F\cten_{K,\pi}-}$
is multirepresentable, i.e. $\coendf{\funct F^{n}}{}\cong C^{\cten n}$.\end{lem}
\begin{proof}
Similar to \cite[3.8.6]{PAR1}. Let $C=\coend{}F$, $\delta:\funct F\to\funct F\cten C$
be the universal transformation and denote $\delta_{n}:=\tau\circ\left(\delta^{\cten n}\right):\funct F^{\cten n}\to\funct F^{\cten}\cten C^{\cten n}$
the transformation 
\[
\delta_{n}:\funct F\left(X_{1}\right)\cten\funct F\left(X_{2}\right)\cten\dots\cten\funct F\left(X_{n}\right)\to\funct F\left(X_{1}\right)\cten C\cten\funct F\left(X_{2}\right)\cten C\cten\dots\cten\funct F\left(X_{n}\right)\cten C
\]
\[
\cong\funct F\left(X_{1}\right)\cten\funct F\left(X_{2}\right)\cten\dots\cten\funct F\left(X_{n}\right)\cten C^{\cten n}
\]
Let $M\in\underline{CHLCTVS}$ and $\phi:\funct F^{\cten n}\to\funct F^{\cten}\cten M$
be a natural transformation. Since $C$ is a CT-$\cten$-coalgebra,
there exists a system $\left\{ C_{i}\right\} _{i\in\mathbb{N}}$ of
Banach $\cten$-coalgebras, such that $C\cong\ilim C_{i}$ is a compact
inductive limit. Each $C_{i}$ is a right Banach $\cten$-comodule
over $C$ via cannonical injections $C_{i}\hookrightarrow C$. For
every multi-index $\bar{i}$ define a map 
\[
\phi_{\bar{i}}:C_{i_{1}}\cten C_{i_{2}}\cten\dots\cten C_{i_{n}}=\funct F\left(C_{i_{1}}\right)\cten\funct F\left(C_{i_{2}}\right)\cten\dots\cten\funct F\left(C_{i_{n}}\right)\to M
\]
\[
\phi_{\bar{i}}:=\left(\epsilon_{C}^{\cten n}\cten\id M\right)\circ\phi\left(C_{i_{1}}\cten C_{i_{2}}\cten\dots\cten C_{i_{n}}\right).
\]
 Similar to (\cite[1.1.32]{EM}), we have a topological isomorphism 
\[
\ilim C_{i_{1}}\cten C_{i_{2}}\cten\dots\cten C_{i_{n}}\cong C^{\cten n}
\]
and, via the universal property of limits, $\phi_{\bar{i}}$ induce
the map $\tilde{\phi}:C^{\cten n}\to M$. The rest of the proof goes
exactly as in \cite[3.8.6]{PAR1}.
\end{proof}
Now we can apply proposition \ref{prop:The-functor-Nat-properties}
to prove reconstruction for bialgebras.
\begin{prop}
Let $C$ be a CT-$\cten$-bialgebra \emph{(}and thus $\comod{\underline{\Ban}}-C$
is a monoidal category\emph{)}. Then $\nat{\underline{\Ban}}{\funct F}{\funct F\Cten C-}$
is multirepresentable by $C$ and 
\[
\coendf{\funct F}{}\cong C
\]
 is the isomorphism of bialgebras.
\end{prop}

\section{Reconstruction for Banach coalgebras}

If one wants to reconstruct a Banach coalgebra, one has to restrict
to finite-dimensional comodules. The reason is because only for finite-dimensional
Banach spaces $X,Y\in\Ban$ a cohomomorphism object $\cohomm[\Ban]XY$
exists (same reason as for $K$-vector spaces). Another problem is
that the category $\Ban$ is not cocomplete, so for a functor $\funct F$
into finite-dimensional Banach spaces $\coend{}{\funct F}$ might
not exist. We will modify our previous construction to handle this
situation.

\subsection{Relative limits and colimits.}

Recall that 
\begin{itemize}
\item for a small category $J$ a $J$-diagram in a category $\cal C$ is
a functor $\funct D:J\to C$;
\item for a $c\in\mc C$ the constant functor $\Delta\left(c\right)\in{\rm Funct}\left(J,\cal C\right)$
is a functor with values $\Delta\left(c\right)\left(j\right):=c$
and $\Delta\left(c\right)\left(j\to j'\right):=\id c$;
\item a cone over $\funct D$ is a natural transformation $\phi:\Delta\left(c\right)\to\funct D$
for some $c\in\mc C$;
\item a cocone over $\funct D$ is a natural transformation $\psi:\funct D\to\Delta\left(c\right)$;
\item a morphism of (co)cones is natural transformation of constant functors
$\Delta\left(c\right)\to\Delta\left(c'\right)$, i.e. a morphism $c\to c'$
in $\mc C$. 
\item cones and cocones over $\funct D$ form categories $Cones\left(\funct D\right)$
and $Cocones\left(\funct D\right)$ correspondingly;
\item a limit $\plim\funct D$ of $\funct D$ is a terminal object in $Cones\left(\funct D\right)$
(if it exists);
\item a colimit $\ilim\funct D$ of $\funct D$ is an initial object in
$Cocones\left(\funct D\right)$.
\end{itemize}
Now let $\mc E-cones\left(\funct D\right)\subset Cones\left(\funct D\right)$
($\mc E-cocones\subset Cocones\left(\funct D\right)$) be a subcategory,
which belongs to a certain collection $\mc E\subset{\rm Mor}\left({\rm Funct}\left(J,{\cal C}\right)\right)$.
\begin{defn}
An $\cal E$-limit of $\funct D$ is a terminal object in $\mc E-cones\left(\funct D\right)$,
which we denote by $\mc E-\plim\funct D$.

Similarly, an $\cal E$-colimit of $\funct D$ is an initial object
in $\mc E-cocones\left(\funct D\right)$, which we denote by $\mc E-\ilim\funct D$.\end{defn}
\begin{rem}
One can give more general definitions, which we omit here for clarity.\end{rem}
\begin{example}
Let $\cal C=\Ban$ and let $\funct F,\funct G\in{\rm Ob}\left({\rm Funct}\left(J,\Ban\right)\right)$.
We say that a natural transformation $\phi:\funct F\dot{\to}\funct G$
is bounded if $\exists C_{\phi}>0$ such that $\norm[\phi\left(j\right)]<C_{\phi}$
for all morphisms $j\in J$. Denote the set of bounded natural transformations
$\funct F\dot{\to}\funct G$ by $BNat\left(\funct F,\funct G\right)$
and the collection of all bounded natural transformations by $BNat\subset{\rm Mor}\left({\rm Funct}\left(J,\Ban\right)\right)$.

Clearly identity transformation $\id{\funct F}:\funct F\to\funct F$
is bounded for every $\funct F\in{\rm Ob}\left({\rm Funct}\left(J,\Ban\right)\right)$
and the composition of bounded natural transformations is bounded.
It is also clear that any natural transformation between constant
functors is bounded. Thus for any diagram $\funct D:J\to\Ban$ we
can form categories of bounded cones $BCones\left(\funct D\right)$
and bounded cocones $BCocones\left(\funct D\right)$. The corresponding
relative limits and colimits will be called Banach limit and colimit
and denoted as $B\plim\funct D$ and $B\ilim\funct D$. Consider some
examples. 

For every functor $\funct F\in{\rm Ob}\left({\rm Funct}\left(J,\Ban\right)\right)$
we first consider the Banach direct product 
\[
\prod\funct F:=\left\{ f\in\prod\funct F\left(j\right)|\sup_{j}\norm[f_{j}]<\infty\right\} 
\]
with projections $\pi_{j}:\prod\funct F\to\funct F\left(j\right)$.
$\prod\funct F$ is a Banach space w.r.t. supremum-norm. The intersection
of all kernels $\ker\left(\pi_{j}-\funct F\left(j'\to j\right)\circ\pi_{j'}\right)$,
for all morphisms $j'\to j$, is a closed subspace of $\prod\funct F$
(possibly a zero subspace) and it satisfies the universal property
of $B\plim\funct D$.

To form Banach colimit of $\funct F$ we first consider the Banach
direct sum 
\[
\sum\funct F:=\left\{ f\in\prod\funct F|\forall\epsilon>0:\,{\rm card}\left\{ j:\,\norm[f_{j}]>\epsilon\right\} <\infty\right\} ,
\]
(completion of the algebraic direct sum $\bigoplus_{j\in J}\funct F\left(j\right)$
in $\prod\funct F$) with injections $\phi_{j}:\funct F\left(j\right)\to\sum\funct F$.
$\sum\funct F$ is a closed subspace of $\prod\funct F$. Let $I$
be the closure in $\sum\funct F$ of the span of the elements $\phi_{j'}\left(x\right)-\phi_{j}\circ\funct F\left(j'\to j\right)\left(x\right)$
for all morphisms $j'\to j$ and for all $x\in\funct F\left(j'\right)$.
The quotient $\sum\funct F/I$ is a Banach colimit of $\funct F$.
\end{example}

\subsection{Bounded coends}

In this section we show that one can do reconstruction in the category
of Banach spaces if one work in relative setting. First we need to
make the corresponding definitions.
\begin{defn}
Let $\mc B$ be a category.
\begin{itemize}
\item Let $\funct G:\cal B\times\cal B^{op}\to\Ban$ be a bifunctor. A (co)wedge
$\mu$ is called \emph{bounded }if $\exists C_{\mu}>0$ such that
$\norm[\mu_{c}]<C_{\mu}$ for all $c\in C$.
\item $B{\rm end}\left(\funct G\right)$ is a terminal object in the category
of bounded wedges over $\funct G$. 
\item $B\coendf{\funct G}{}$ is an initial object in the category of bounded
cowedges over $\funct G$.
\item For a functor $\funct F:\mc B\to\Ban^{fd}\subset\Ban$ valued in finite-dimensional
Banach spaces $\Ban^{fd}$, we define 
\[
B\coendf{\funct F}{}:=B\coendf{{\rm Cohom}\circ\left(\funct F\times\funct F^{op}\right)}{}.
\]

\end{itemize}
\end{defn}
\begin{lem}
Similar to the case of $Nat$ and ${\rm coend}$ one can prove the
following results.
\begin{itemize}
\item there is a one-to-one correspondence between elements of $B{\rm Nat}_{\mc{}}\left(\funct F,\funct F\Cten CM\right)$
and bounded cowedges from ${\rm Cohom}\circ\left(\funct F\times\funct F^{op}\right)$
to $M$;
\item if $B\coendf{\funct F}{}$ exists, we have an isomorphism 
\[
BNat\left(\funct F,\funct F\cten M\right)\cong\Ban\left(B\coendf{\funct F}{},M\right);
\]

\item under the isomorphism above, from the identity morphism $\id{B\coendf{\funct F}{}}$
we get a bounded universal natural transformation $\delta_{\funct F}:\funct F\to\funct F\cten B\coendf{\funct F}{}$. 
\end{itemize}
\end{lem}
\begin{rem}
Let $\funct F:\mc B\to\Ban$ be a functor. If $\mc B$ is small and
the objects $\cohomm{\funct F\left(X\right)}{\funct F\left(X\right)}$
exist (i.e. if $\funct F\left(X\right)$ are finite dimensional) ,
then the $B\coendf{\funct F}{}$ will exist and represent $BNat\left(\funct F,\funct F\cten-\right)$.
So one can proceed similar to \cite{PAR2} and define $B\coendf{\funct F}{}$
as the representing object for $BNat\left(\funct F,\funct F\cten-\right)$. \end{rem}
\begin{prop}
Let $B\coendf{\funct F}{}$ exist. Then
\begin{enumerate}
\item $B\coendf{\funct F}{}$ is a Banach $\cten$-coalgebra: 

\begin{itemize}
\item the comultiplication $\cm_{B\coendf{\funct F}{}}$ corresponds to
the transformation $\left(\delta_{\funct F}\cten\id M\right)\circ\delta_{\funct F}$; 
\item the counit $\epsilon_{B\coendf{\funct F}{}}$ corresponds to the transformation
$\id{\funct F}:\funct F\to\funct F\cten K\cong\funct F$;
\end{itemize}
\item $\funct F\left(X\right)$ is a right Banach $\cten$-comodule over
$B\coendf{\funct F}{}$ via $\rho_{\funct F\left(X\right)}:=\delta_{\funct F}\left(X\right)$
for every $X\in\mc B$;
\item $\funct F\left(\phi\right)$ is a morphism of $B\coendf{\funct F}{}$-$\cten$-comodules
for every morphism $\phi\in{\rm Mor}\left(\mc B\right)$;
\item if $\mc B$ is monoidal, $\funct F$ is also be monoidal and $BNat\left(\funct F,\funct F\cten-\right)$
is multirepresentable, then $B\coendf{\funct F}{}$ is a Banach $\cten$-bialgebra;

\begin{lyxlist}{00.00.0000}
\item [{Now let $\forall X\in\mathcal{B}:$ ${\rm dim}\funct F\left(X\right)<\infty$ . Then}]~
\end{lyxlist}
\item if $\mc B$ is monoidal and $\funct F$ is also be monoidal then $B\coendf{\funct F}{}$
is a Banach $\cten$-bialgebra;
\item if $\mc B$ is rigid then $B\coendf{\funct F}{}$ is a Banach Hopf
$\cten$-algebra.
\end{enumerate}
\end{prop}
\begin{defn}
Let $\cal C_{0}\subset\Ban$ be a full monoidal subcategory and $I$
be a poset. The Banach $\cten$-coalgebra $C$ is called $\cal C_{0}-I$\emph{-generated}
if the following holds:
\begin{enumerate}
\item $C$ is a Banach colimit of an $I$-diagram of objects $C_{i}\in\cal C_{0}$;
\item all morphisms $X\cten j_{i}\cten M:X\cten C_{i}\cten M\to X\cten C\cten M$
are monomorphisms in $\cal C$, where $X\in\cal C_{0}$, $M\in\Ban$
and $j_{i}:C_{i}\to C$ are the monomorphisms from the colimit diagram;
\item every $C_{i}$ is a Banach $\cten$-subcoalgebra of $C$ via $j_{i}:C_{i}\to C$;
\item If $\left(P,\rho_{P}\right)$ is a (right) Banach $\cten$-comodule
over $C$ and $P\in\cal C_{0}$ then $\exists$ $i$ and $\rho_{P,i}:P\to P\cten C_{i}$
such that $\rho_{p}=\left(\id P\cten j_{i}\right)\circ\rho_{P,i}$.
\end{enumerate}
\end{defn}
\begin{prop}
(Reconstruction theorem) Let $\cal C_{0}$ be a full braided monoidal
subcategory of $\Ban$.
\begin{enumerate}
\item Let $C$ be a $\cal C_{0}-I$-generated Banach $\cten$-coalgebra
and $\funct F:\comod{\cal C_{0}}-C\to\cal C_{0}\subset\Ban$ be the
forgetful functor. Then we have an isomorphism of coalgebras $B\coendf{\funct F}{}\cong C$;
\item Let $C$ also be a Banach $\cten$-bialgebra \emph{(}and thus $\comod{\cal C_{0}}-C$
is a monoidal category\emph{)}. Then $B\nat{}{\funct F}{\funct F\cten-}$
is multirepresentable by $C$ and $B\coendf{\funct F}{}\cong C$ is
the isomorphism of bialgebras.
\end{enumerate}
\end{prop}
\begin{proof}
Let $\left\{ C_{i}\right\} _{i\in I}$ be the generating diagram for
$C\cong\ilim C_{i}$ and $\phi:\funct F\to\funct F\cten M$ be a bounded
natural transformation. Define the maps $\tilde{\phi}_{i}:C_{i}\to M$
as $\tilde{\phi}_{i}:=\left(\epsilon_{C_{i}}\cten\id M\right)\circ\phi\left(C_{i}\right)$.
Since the transformation $\phi$ is bounded, the maps $\left\{ \tilde{\phi}_{i}\right\} _{i\in I}$
form a bounded cocone over $\left\{ C_{i}\right\} _{i\in I}$, and
thus define a map $\tilde{\phi}:C\to M$. The rest of the proof of
$(1)$ goes as in \cite[3.8.4]{PAR1}.

The multirepresentability of $B\nat{}{\funct F}{\funct F\cten-}$
is proven similar to lemma \ref{banCTmultirepr} (since Banach tensor product is a left adjoint, it is cocontinuous and preserves small colimits, \cite{Mich},1.4). 
Then the bialgebra
structure is reconstructed in the standard way (see \cite[2.3.7]{SCH}).
\end{proof}

\section*{appendix}

\subsection*{Comodule structure on \textmd{$\protect\cohom[C]XY$.}}

In applications reconstructing bialgebra structure might be sufficient
due to the uniqueness of the antipode. However in the recognition
theorem one would like to have conditions, which define Hopf algebra
structure on $\coendf{\funct F}{\mc C}$. Our conjecture is that one
also can reconstruct Hopf algebra structure under assumptions weaker
that rigidity. Currently we cannot prove it in full generality and
this is the reason why in the title of this paper we only put ``reconstruction
for bialgebras''.

Suppose we have a Hopf algebra $H$ in a category $\mc C$ and $\comod{\mc C_{0}}-H$
be the category right $H$-comodules, which are the objects of the
rigid subcategory $\mc C_{0}\subset\mc C$. Then for every $X\in\comod{\mc C_{0}}-H$
we have a comodule structure on $X^{*}$ via the antipode of $H$
and this gives the rigid structure on $\comod{\mc C_{0}}-H$. Thus
the category $\comod{\mc C_{0}}-H$ is coclosed and for any $X,Y\in\comod{\mc C_{0}}-H$
we have an equality in $\mc C_{0}$ 
\[
\coend{{\rm \comod{\mc C_{0}}-H}}{X,Y}=\coend{C_{{\rm 0}}}{X,Y}=Y^{*}\ten X.
\]
We anticipate that the first part of this equality holds for general
coclosed categories. Here we only explain how to give $\cohom[C]XY$
a comodule structure.

Let now $\mc C_{0}\subset\mc C$ be a subcategory coclosed in $\mc C$. 

Let $C\in\cal C$ be a coalgebra and $X\in\cal C_{0}$ is a right
$C$-comodule. Then the coaction $\rho_{X}:X\to X\Cten CC$ induces
the map 
\[
\rho_{\cohom[C]XY}^{r}:=\cm_{X,Y,X\Cten CC,\rho_{X}}:\cohom[C]XY\to\cohom[C]XY\Cten CC
\]
 for every $Y\in\mc C_{0}$ via diagram \begin{center}
\begin{tikzpicture}[node distance=2cm, auto]    
\node (x) {$X$};
\node (xc) [above of=x] {$X\Cten C C$};
\node (ycoh) [right= and 3cm of x] {$Y\Cten C\cohom[C]XY $};   
\node (ycohc) [above of=ycoh] {$Y\Cten C\cohom[C]XY\Cten C C$};   
\draw[->] (x) to node {$\rho_X$} (xc);
\draw[->] (xc) to node {$\coev {X,Y}\Cten C \id{C}$} (ycohc); 
\draw[->] (x) to node [swap] {$\coev {X,Y}$} (ycoh); 
\draw[->] (ycoh) to node [swap] {$\mathrm{id}_Y\Cten C \rho_{\cohom[C]XY}$} (ycohc); 
\end{tikzpicture} 
\end{center}One can check that $\rho_{\cohom[C]XY}^{r}$ satisfy the axioms of
the right $C$-comodule coaction. In a similar way the coaction $\rho_{Y}:Y\to Y\Cten CC$
induces the map 
\[
\tilde{\rho}_{\cohom[C]XY}^{l}:\cohom[C]XY\to C\Cten C\cohom[C]XY
\]
via diagram \begin{center}
\begin{tikzpicture}[node distance=2cm, auto]    
\node (x) {$X$};
\node (xc) [above of=x] {$Y\Cten C\cohom[C]XY $};
\node (ycoh) [right= and 4cm of x] {$Y\Cten C\cohom[C]XY $};   
\node (ycohc) [above of=ycoh] {$Y\Cten C C\Cten C\cohom[C]XY$};   
\draw[->] (x) to node {$\coev {X,Y}$} (xc);
\draw[->] (xc) to node {$\rho_Y\Cten C \id{\cohom[C]XY}$} (ycohc); 
\draw[->] (x) to node [swap] {$\coev {X,Y}$} (ycoh); 
\draw[->] (ycoh) to node [swap] {$\mathrm{id}_Y\Cten C \rho_{\cohom[C]XY}'$} (ycohc); 
\end{tikzpicture} 
\end{center}One can check that $\tilde{\rho}_{\cohom[C]XY}^{l}$ satisfy the axioms
of the left $C^{cop}$-comodule coaction. In case $C=H$ is a Hopf
algebra in $\mc C$ we can turn $\tilde{\rho}_{\cohom[C]XY}^{l}$
into a right $H$-comodule coaction 
\[
\rho_{\cohom[C]XY}^{l}:=\left(\id{\cohom[C]XY}\Cten CS_{H}\right)\circ\tau\circ\tilde{\rho}_{\cohom[C]XY}^{l}.
\]

Combining $\rho_{\cohom[C]XY}^{l}$ and $\rho_{\cohom[C]XY}^{r}$
(similar to the tensor product of $X^{*}$ and $X$ in a rigid category),
we define the map 
\[
\rho_{\cohom[C]XY}:=\left(\id{\cohom[C]XY}\Cten Cm_{H}\right)\circ\left(\rho_{\cohom[C]XY}^{l}\Cten C\id H\right)\circ\rho_{\cohom[C]XY}^{r}
\]
which satisfy the axioms of the right $H$-comodule coaction. Under
this comodule structure, the coevaluation map 
\[
\coev{X,Y}:X\to Y\Cten C\cohom[C]XY
\]
becomes a morphism of right $H$-comodules.
By dualizing the argument in \cite{STR}, proposition 10.1, one can show that coaction morphisms are morphisms of $H$-comodules, making $\cohom[C]XY$ into $\cohom[\comod{\mc C}-H]XY$.

\subsection*{Acknowledgments.}

Although the question, answered in this paper, was raised quite a
while ago, the active phase of this work was completed during my stay
in the University of Science and Technology of China as a post-doc.
I would like to thank the institution, the department, the Wu Wenjun
CAS Key Laboratory of Mathematics and its director professor Sen Hu,
and especially professor Yun Gao for hospitality, stimulation, support
and excellent research conditions. Needless to say that all possible
errors and inaccuracies in this paper are solely my fault. I also
thank Peter Schauenburg for explaining me one vague point in \cite{SCH}.

I first started thinking about possible version of Tannaka duality
in nonarchimedean setting when I was taking a course of Alex Rosenberg
on reconstruction theorems. Unfortunately, the key points of the present
work became clear to me only recently, when it was already too late
to discuss it with him. This paper is devoted to his memory.

\end{document}